\documentclass{amsart}
\usepackage[margin=2.5cm]{geometry}
\usepackage[foot]{amsaddr}
\usepackage[utf8]{inputenc}
\usepackage[T1]{fontenc}
\usepackage{amsmath,amssymb}
\usepackage{graphicx}
\usepackage[colorinlistoftodos,prependcaption,textsize=tiny]{todonotes}
\usepackage{color}
\usepackage[
		backend=biber,
		style=alphabetic,
		url=false, 
		isbn = false,
	datamodel=eprint-hal,
]{biblatex}

\usepackage{hyperref}

\DeclareFieldFormat{hal}{%
	\mkbibacro{HAL}\addcolon\space
	\ifhyperref
	{\href{https://hal.archives-ouvertes.fr/#1}{\nolinkurl{#1}}}
	{\nolinkurl{#1}}}

\DeclareFieldAlias{eprint:hal}{hal}
\DeclareFieldAlias{eprint:HAL}{eprint:hal}

\renewbibmacro*{eprint}{%
	\printfield{hal}%
	\newunit\newblock
	\iffieldundef{eprinttype}
	{\printfield{eprint}}
	{\printfield[eprint:\strfield{eprinttype}]{eprint}}}

\AtEveryBibitem{\clearfield{pages}} 
\AtEveryBibitem{\clearfield{urldate}} 
\AtEveryBibitem{\clearlist{language}}
\DeclareDatamodelEntryfields{hal}

\usepackage{hyperref} 
\hypersetup{colorlinks=true,citecolor=blue}

\newcommand{\defi}{:=}
\newcommand{\st}{ : }

\newcommand{\R}{\mathbb{R}}

\newcommand{\E}{\mathbb{E}}

\newcommand{\norm}   [1] {\left\Vert #1 \right\Vert}

\newcommand{\muW}{\mu^{\mathbf{w}}}
\newcommand{\pW}{\mathbf{p}^{\mathbf{w}}}
\newcommand{\Vn}{V_{null}}

\newcommand{\te}[1]{ t_\infty(#1)}
\newcommand{\tee}{ t_\infty}

\newcommand{\Ew}{\mathbf E_w}
\renewcommand{\P}{\mathbb{P}}
\newcommand{\CP}{{\boldsymbol{\mathcal{P}}}}
\newcommand{\CPf}[2]{\bigl(\boldsymbol{\mathcal{P}}^{I}\bigr)_{I \in [#1,#2]}}
\newcommand{\CR}{\mathbf{R}}

\newcommand{\I}{\kappa}
\newcommand{\vr}  {\bar v}
\newcommand{\wb}  {\bar w}
\newcommand{\Gb}  {G^{\textrm{\upshape b}}}
\newcommand{\mU}{\mathbf U}
\newcommand{\vmin}{v_{\textrm{\tiny\upshape min}}}

\newcommand{\vsep}{v_{\textrm{\upshape\tiny sep}}}
\newcommand{\wsep}{w_{\textrm{\upshape\tiny sep}}}
\newcommand{\Iempty}{{\mathbf I_\emptyset}}
\newcommand{\Ireg}{{\mathbf I_{reg}}}

\renewcommand{\Phi}{\varphi}

\newcommand{\bfw}{\mathbf{w}}
\newcommand{\bfU}{\mathbf{U}}
\newcommand{\bfG}{\mathbf{G}}

\newcommand{\indic}{\mathbf 1}

\newcommand{\rmd}  {{\textrm{\upshape d}}}

\newcounter{lem}
\newtheorem*{asumptionstar}{Assumption}
\newtheorem{asumption}[lem]{Assumption}
\newtheorem{proposition}[lem]{Proposition}
\newtheorem{theorem}[lem]{Theorem}
\newtheorem{lemma}[lem]{Lemma}

\newtheorem{definition}[lem]{Definition}
\newtheorem{corollary}[lem]{Corollary}
\newtheorem{remark}[lem]{Remark}

\newcommand{\vertiii}[1]{{\left\vert\kern-0.25ex\left\vert\kern-0.25ex\left\vert #1
    \right\vert\kern-0.25ex\right\vert\kern-0.25ex\right\vert}}

\graphicspath{{figures}{../figures}}

\author{Romain Veltz} 
	
\address{Université Côte d’Azur, Inria, France}
\date{\today}
\title[Mean-field limit of interacting 2d integrate-and-fire neurons]{Analysis of a mean-field limit of interacting two-dimensional nonlinear integrate-and-fire neurons}

\begin{document}

\begin{abstract}
	We study the solutions of a McKean-Vlasov stochastic differential equation (SDE) driven by a Poisson process. In neuroscience, this SDE models the mean field limit of a system of $N$ interacting excitatory neurons with $N$ large. 
    Each neuron spikes randomly with rate depending on its membrane potential. At each spiking time, the neuron potential is reset to the value $\vr$, its adaptation variable is incremented by $\wb$ and all other neurons receive an additional amount
	$J/N$ of potential after some delay where $J$ is the connection strength. Between jumps, the neurons drift according to some two-dimensional ordinary differential equation with explosive behavior.
	We prove the existence and uniqueness of solutions of a heuristically derived mean-field limit of the system when $N\to\infty$. We then study the existence of stationary distributions and provide several properties (regularity, tail decay, etc.) based on a Doeblin estimate using a Lyapunov function. Numerical simulations are provided to assess the hypotheses underlying the results. 
	
\end{abstract}

\maketitle	
\section{Introduction}
Our goal is to establish existence and uniqueness of solutions to the heuristically derived mean-field limit of models of networks of stochastic spiking neurons introduced in \cite{aymard_mean-field_2019}. Since the dynamics of the neurons lack an analytical expression, we aim to provide testable assumptions at the level of the vector field.

\subsection{Description of the model}

For a given network size $N\in\mathbb N^*$, we consider a piecewise deterministic Markov process \cite{davis_markov_1993} (PDMP) $X_t^N=(V_t^{1,N},\cdots,V_t^{N,N},W_t^{1,N},\cdots,W_t^{N,N} )\in\mathbb R^{2N}$. For $i=1,\cdots,N$, $V_t^{i,N}$  represents the membrane potential of neuron $i$ and  $W_t^{i,N}$ its adaptation variable. Neuron $i$ emits spikes at random times with spiking rate $\lambda(V_t^{i,N})$. When it emits a spike at time $t_s$, its potential is reset
$V^{i,N}_{t_s^+} = \vr$, its adaptation variable is increased by an amount $\wb\geq0$, $W^{i,N}_{t_s^+}=W^{i,N}_{t_s^-}+\wb$, and the potential of the other neurons is incremented by an amount $\frac JN$ after a delay $D>0$:
\[
\forall j\neq i,\quad V^{j,N}_{(t_s+D)^+} = V^{j,N}_{(t_s+D)^-}+\frac J N.
\]
Between two spikes, the neuron variables $(V^{i,N}_{t}, W^{i,N}_{t})$  evolve according to the ordinary differential equation (ODE)

\begin{equation*}
	\left\{
	\begin{aligned}
		\dot V_t^{i,N} &= F(V_t^{i,N} )- W^{i,N}_t  + I \\
		\dot W^{i,N}_t &= V_t^{i,N} - W^{i,N}_t.
	\end{aligned}
	\right.
	\label{eq:ode}
\end{equation*}
The process $X$ can be re-written as a SDE driven by Poisson measures. Let $\left(\mathbf{N}^{i}(d u, d z)\right)_{i=1, \ldots, N}$ be a family of $N$ independent Poisson measures on $[-D,\infty)\times\mathbb R_+$ with intensity measure $du dz$. Let $\Big (V_0^{i,N},W_0^{i,N}\Big )_{i=1,\cdots,N}$ be a family of $N$ independent random variables in $\mathbb R^2$ with law $\mu_0$, and independent of the Poisson measures. Then $X^N$ is a \textit{c\`adl\`ag} process solution of the SDEs $\forall i=1,\cdots,N$:
\[
\begin{aligned}
	\quad V_{t}^{i, N}&= V_{0}^{i, N}+\int_{0}^{t} \left(F(V_u^{i,N} )- W^{i,N}_u  + I \right)d u+\frac{J}{N} \sum_{j \neq i} \int_{[-D,t-D]\times\mathbb{R}_{+}} \mathbf{1}_{\left\{z \leq \lambda\left(V_{u-}^{j, N}\right)\right\}} \mathbf{N}^{j}(d u, d z) \\
	&\qquad\qquad\qquad\qquad+\int_{[0,t]\times\mathbb{R}_{+}} \left(\vr -V_{u-}^{i, N}\right) \mathbf{1}_{\left\{z \leq \lambda\left(V_{u-}^{i, N}\right)\right\}} \mathbf{N}^{i}(d u, d z),\\
	\quad W_{t}^{i, N}&= W_{0}^{i, N}+\int_{0}^{t} \left(V_u^{i,N} - W^{i,N}_u\right) d u+\int_{[0,t]\times\mathbb{R}_{+}} \wb\, \mathbf{1}_{\left\{z \leq \lambda\left(V_{u-}^{i, N}\right)\right\}} \mathbf{N}^{i}(d u, d z) 
\end{aligned}
\]
and $X^N_t=(V_0,W_0)$ for $t\leq 0$.

When $N$ converges to infinity, we expect \cite{de_masi_hydrodynamic_2015,fournier_toy_2016} that $X_t^N$ converges in law to the solution of the McKean-Vlasov SDE:
\begin{equation}\label{eq:MKV}\tag{MKV}
	\begin{aligned}
		V_{t}^{nl}&= V_{0}+\int_{0}^{t} \left[F(V_u^{nl} )- W^{nl}_u  + I+J\E\lambda\left(V_{u-D}^{nl}\right) \right]d u+\\&\qquad\qquad\qquad\qquad\qquad\qquad\qquad\qquad
		\int_{[0,t]\times\mathbb{R}_{+}}\left(\vr -V_{u-}^{nl}\right) \mathbf{1}_{\left\{z \leq \lambda\left(V_{u-}^{nl}\right)\right\}} \mathbf{N}(d u, d z) ,\\
		\quad W_{t}^{nl}&= W_{0}+\int_{0}^{t} \left(V_u^{nl} - W^{nl}_u\right) d u+\int_{[0,t]\times\mathbb{R}_{+}}\wb\, \mathbf{1}_{\left\{z \leq \lambda\left(V_{u-}^{nl}\right)\right\}} \mathbf{N}(d u, d z)
	\end{aligned}
\end{equation}
where $\mathbf N$ is a Poisson measure on $\mathbb R_+\times\mathbb R_+$ with intensity measure $du dz$, $\mathcal L(V_0,W_0)=\mu_0$ and $X^{nl}_t=(V_0,W_0)$ for $t\leq 0$.  The variables $\mathbf N$ and $(V_0,W_0)$ are independent.

Equation \eqref{eq:MKV} is a mean-field equation and is the current object of interest. Note that \eqref{eq:MKV} involves the law of the solution in the term $\E\lambda\left(V_{u-D}^{nl}\right)$.
In addition to \eqref{eq:MKV}, for $\I\in C(\R,\R_+)$, we  consider the non-autonomous process $X^{s,\mu_0,\I}=(V^{s,\mu_0,\I},W^{s,\mu_0,\I})$ solution of the  ``linearized'' equation with initial condition $X_{s}^{s,\mu_0,\I}\sim \mu_0$ independent of $\mathbf N$:
\begin{equation}\label{eq:EDSL}\tag{I-SDE}
	\begin{aligned}
		V_{t}^{s,\mu_0,\I}&= V_{s}^{s,\mu_0,\I}+\int_{s}^{t} \left[F(V_u^{s,\mu_0,\I} )- W^{s,\mu_0,\I}_u  + I+ \I(u) \right]d u\quad+
		\\&\qquad\qquad\qquad\qquad\int_{[s,t]\times\mathbb{R}_{+}} \left(\vr -V_{u-}^{s,\mu_0,\I}\right) \mathbf{1}_{\left\{z \leq \lambda\left(V_{u-}^{s,\mu_0,\I}\right)\right\}} \mathbf{N}(d u, d z) ,\\
		\quad W_{t}^{s,\mu_0,\I}&= W_{s}^{s,\mu_0,\I}+\int_{s}^{t} \left(V_u^{s,\mu_0,\I} - W^{s,\mu_0,\I}_u\right) d u+\int_{[s,t]\times\mathbb{R}_{+}} \wb\, \mathbf{1}_{\left\{z \leq \lambda\left(V_{u-}^{s,\mu_0,\I}\right)\right\}} \mathbf{N}(d u, d z) .
	\end{aligned}
\end{equation}
When $\mu_0=\delta_x$, we write the associated process $X^{s,x,\I}$. We also drop the dependency on $s,\mu_0,\I$ when possible.

\subsection{Biological and modeling background}

Although the above particle system is a toy model, it is inspired by biology.
A neuron is a specialized cell type of the central nervous system \cite{kandel_principles_2013} comprised of subcellular domains: dendrite, soma and axon. The neurons are connected to each others by synapses which connect axons to dendrites of different neurons. The neurons communicate using electrical impulses encoded in their membrane electrical potential. When the difference of electrical potential $V_t$ across the membrane of their soma is high enough, a sequence of action potentials (or spikes) is produced and the somatic membrane potential returns to a resting value $\vr$. These electrical impulses are short $2-5ms$. The action potentials then propagate along the axons until they reach the synapses, connections to the other neurons. When an action potential reaches a synapse, it triggers a local change of the membrane potential of the postsynaptic side of the synapse which itself induces a local change of the somatic membrane potential of the postsynaptic neuron (we ignore propagation of electric signals in the dendrite). This process takes time and we model it with a small constant delay $D>0$ on the order of $5ms$. Also, we assume that all neurons are connected to each others and each spike produces the same local change $\frac JN$ in the postsynaptic neurons where $J$ is the connection strength. This is of course a strong idealization.

The generation of spikes at the soma originates from the fluxes of different ions species across the cellular membrane, a process described by the four dimensional deterministic Hodgkin-Huxley (HH) model \cite{hodgkin_quantitative_1952} which has been simplified to a two dimensional model \cite{gerstner_neuronal_2014,brette_adaptive_2005}. The noise in the spiking dynamics is modeled here with a \textit{firing intensity} $\lambda(V_t)$ which depends on the current membrane potential $V_t$ of the neuron. This is inspired by the one dimensional \textit{escape rate} models or the \textit{generalized integrate-and-fire} (GIF) models\index{GIF, generalized integrate-and-fire model} but it has never been studied with two dimensional model models except in \cite{aymard_mean-field_2019}. 

We refer to \cite{brillinger_maximum_1988,gerstner_time_1995,gerstner_neuronal_2014,schwalger_towards_2017} and the references therein for a description of the one dimensional GIF model. On a side note, the GIF model have been shown to be relevant for modeling real data \cite{mensi_parameter_2012}. See also \cite{jahn_motoneuron_2011} for fitting data to a jump diffusion integrate and fire.

\

One long term goal of this work is to link the macroscopic dynamics of the networks to the microscopic elements. To this end, we need a model of neuron with a rich set of dynamical (spiking) behaviors \cite{izhikevich_dynamical_2007}.
Most mean field models in the mathematical neuroscience community assume that the neurons have simple dynamics with affine scalar vector field ; neurons with 2d dynamics have rarely been studied (see recent review \cite{carrillo_nonlinear_2025}). This is the case of neurons modeled with Hawkes processes \cite{chevallier_microscopic_2015,ditlevsen_multi-class_2017}, with the integrate and fire model \cite{caceres_analysis_2011,carrillo_classical_2013,delarue_global_2015,delarue_particle_2015,dumont_mean-field_2020} and with PDMPs \cite{de_masi_hydrodynamic_2015, fournier_toy_2016, robert_dynamics_2016, cormier_long_2020,drogoul_exponential_2021}. 

Some notable exceptions are studies \cite{riedler_limit_2012, bossy_clarification_2015} with the 4d HH dynamics \cite{gerstner_neuronal_2014} or the 2d FitzHugh-Nagumo ones \cite{mischler_kinetic_2016,crevat_rigorous_2019,lucon_periodicity_2021}, the studies by Nicolas and Campbell \cite{nicola_mean-field_2013,nicola_one-dimensional_2015}, by Destexhe and collaborators \cite{zerlaut_modeling_2018, di_volo_biologically_2019} or the ones based on jump processes \cite{galves_system_2020,schmutz_mean-field_2022, duval_interacting_2022}. Finally, we mention the 2d voltage-conductance kinetic model \cite{perthame_voltage-conductance_2013, salort_convergence_2024,sanchez_voltage-conductance_2024}.

The HH model suffers from sophisticated dynamics which makes its parameter tuning difficult. This led to the development of the family of two dimensional nonlinear spiking neuron models \cite{brette_adaptive_2005,izhikevich_dynamical_2007,gerstner_neuronal_2014,touboul_spiking_2009} which reproduces the majority of observed spiking dynamics such as bursting, spike adaptation, etc. In essence, the two dimensional nonlinear spiking models trade the complexity of the HH model for a two dimensional explosive flow and a reset.
Several choices are possible for the nonlinearity $F$, corresponding to different
classical models, such as
\[
\begin{aligned}
F(v) &= v(v  - a),\ a\in \mathbb{R} &&\mbox{               (Izhikevich model \cite{izhikevich_dynamical_2007}),}   \\
F(v) &= e^v + av,\ a\in\R  &&\mbox{             (AdEx model \cite{brette_adaptive_2005})},\\
F(v) &= v^4 + {2\,a v},\ a\in \mathbb{R} &&\mbox{          (quartic model \cite{touboul_spiking_2009})}
\end{aligned}
\]
but they have more or less the same spiking dynamics repertoire \cite{touboul_spiking_2009}.
Our model is related to the ones in \cite{schwalger_towards_2017, di_volo_biologically_2019} which are used to study the dynamics of cortical columns.

One of the first studies of networks of neurons modeled with PDMPs is probably \cite{gerstner_time_1995} and \cite{gerstner_population_2000}. A more biologically relevant way to model the interaction between neurons is to consider conductance based models (see review on mean-field models of networks of neurons \cite{carlu_mean-field_2020} and next section~\ref{section:perspectives}).

\subsection{Modeling perspectives}\label{section:perspectives}
The model can be extended in may ways. A first extension would be to generalize the reset condition to $w\to \alpha w+\wb$ or to a more general distribution.
One could also use the more biologically relevant and recent neuron model \cite{gorski_conductance-based_2021}. One could also change the synapse model using a conductance based one \cite{di_volo_biologically_2019}, in our notations, this would give:
\begin{equation*}
\begin{aligned}
V_{t}^{nl}&= V_{0}+\int_{0}^{t} \left[F(V_u^{nl} )- W^{nl}_u + I+G_u^{nl}\cdot(V_u^{nl}-E) \right]d u+\int_{[0,t]\times\mathbb{R}_{+}}\left(\vr -V_{u-}^{nl}\right) \mathbf{1}_{\left\{z \leq \lambda\left(V_{u-}^{nl}\right)\right\}} \mathbf{N}(d u, d z) ,\\
W_{t}^{nl}&= W_{0}+\int_{0}^{t} \left(V_u^{nl} - W^{nl}_u\right) d u+\int_{[0,t]\times\mathbb{R}_{+}}\wb\, \mathbf{1}_{\left\{z \leq \lambda\left(V_{u-}^{nl}\right)\right\}} \mathbf{N}(d u, d z),\\
G_{t}^{nl}&= G_{0}+\int_{0}^{t} \left(-\tau^{-1} G_u^{nl} + g\E\lambda\left(V_{u-D}^{nl}\right)\right) d u .
\end{aligned}
\end{equation*}
for some constants $g, E, \tau$ and where $G^{nl}$ is the synaptic conductance \cite{gerstner_neuronal_2014}.

\subsection{Goal of the paper}

In the model \eqref{eq:MKV}, we have the freedom of the rate function $\lambda$ which controls the noise level in the network. Owing to the explosive behavior of the vector field, the rate function has to ensure that the neuron spikes before the deterministic explosion time which thus has the tendency to magnify this effect. As $\lambda$ arises in the drift of \eqref{eq:MKV}, there is a possibility that a blowup occurs for \eqref{eq:MKV} when the delays are neglected $D=0$. 

\noindent This can readily be seen for the following SDE, at least formally,
\begin{equation*}
	X_{t}= X_{0}+\int_{0}^{t} \left[F(X_u )+J\E\lambda\left(X_{u}\right) \right]d u+\int_{[0,t]\times\mathbb{R}_{+}}\left(\vr -X_{u-}\right) \mathbf{1}_{\left\{z \leq \lambda\left(X_{u-}\right)\right\}} \mathbf{N}(d u, d z) .
\end{equation*}
The Ito formula gives for $\lambda=F=\exp$
\[
\mathbb E(\lambda(X_t))-\mathbb E{\lambda(X_0)} = \int_0^t J\mathbb E\lambda(X_u)\mathbb E \lambda(X_u)+\lambda(\vr)\mathbb E\lambda(X_u)du
\]
which blows up in finite time.

The first part of the paper is thus to give conditions on $\lambda, F$ for \eqref{eq:MKV} to be well defined for all time $t\geq 0$. We then give conditions for the existence of a stationary distribution. This is done by studying the enclosed Markov chain and by relying on Doeblin estimates \cite{canizo_harris-type_2023} using a Lyapunov function. The invariant distribution of the process $X^{s,\mu_0,\I}$ is then lifted from the one of the enclosed chain, we recover the general expression from \cite{costa_stationary_1990} albeit in a different general setting. 

\subsection{Plan of the paper}
In the next sections~\ref{section:notations} and \ref{section:main-results}, we state our notations and main results. 
In section~\ref{section:ode}, we establish some properties of the deterministic flow.
In section~\ref{section:linear}, we analyze the solution of \eqref{eq:EDSL}  and in section~\ref{section:NL-DDE}, we solve \eqref{eq:MKV}.
In section~\ref{section:DI}, we study the invariant distributions of \eqref{eq:EDSL} and of the enclosed Markov chain.
In section~\ref{section:NLDI}, we study the invariant distributions of \eqref{eq:MKV}.
In section~\ref{section:applications}, we give applications of our results to classical models in neuroscience and discuss extensions of the results based on numerical simulations.
\section{Notations and assumptions}\label{section:notations}

In this work, we assume that the connections between the neurons are excitatory and that the reset condition increases the $w$ variable:

\begin{asumptionstar}
	\begin{equation}
		J>0 \text{ and } \wb>0.
	\end{equation}
\end{asumptionstar}

\subsection{Deterministic part}
The function $F$ in the drift term of the membrane potential dynamics \eqref{eq:MKV} is such that:
\begin{asumption}
	\label{hyp.F}
	
	\
	
	\begin{enumerate}
		\item
		$F$ is of class $C^1$, strictly convex, and $\lim\limits_{-\infty}F'<-3$.
		\item
		There are $\epsilon_F,\alpha_F>0$ such that $F(v)/v^{2+\epsilon_F}\geq\alpha_F$ as $v\to\infty$.
	\end{enumerate}
\end{asumption}
\noindent
This prevents us from using the Izhikevich model \cite{izhikevich_dynamical_2007}. This is not completely troublesome as this model provides a strict subset of the spiking behaviors of the quartic model \cite{touboul_spiking_2009}.

\noindent We write $\Phi_s^{t}(x)=\Phi_s^{t}(v,w)$  the flow of the ODE
\begin{equation}
	\left\{
	\begin{aligned}
		\dot v(t) &= F(v(t) )- w(t)  + I +\I(t)\\
		\dot w(t) &= v(t) - w(t)
	\end{aligned}
	\right.
	\label{eq:micro-unique}
\end{equation}
with initial condition $v(s)=x_1,w(s)=x_2$ and where $\I$ is continuous. For an initial condition $x\in\mathbb R^2$, \eqref{eq:micro-unique} admits a unique maximal solution $\Phi^{t}_s(x)$ for $t\in[t_{-\infty}(s,x),\te{s,x})$ with  $t_{\pm\infty}(s,x)\in\overline\R_{\pm}$.
When the current $\I(t) = \I$ is constant, the flow is homogeneous and only depends on $t-s$: we drop the $s$ dependency in $\Phi,\tee$ in this case.

We use indifferently the notations $x=(x_{1},x_{2})=(v,w)$. For the flow, we  also use the notations $(v_{x}(s,t),w_{x}(s,t))=\Phi^{t}_s(x)$. Finally, whenever the $v$ variable is not mentioned, it means that it is equal to the reset value $\vr$, \textit{e.g.} $v_{w_0}(t) \defi v_{(\vr,w_0)}(t)$.

\begin{definition}\label{def:I0}
	We denote by $\Iempty\subset\R$ the open set of $I$s for which \eqref{eq:micro-unique} with $\I=0$ has no equilibrium.
\end{definition}

\begin{definition}[Regularity set]
	\label{def:S}
	We define  the open set $$\Ireg \defi \{I\st (\vr, \vr) \text{ is not an equilibrium of \eqref{eq:micro-unique}}\text{ with }\I=0\}.$$
\end{definition}

\begin{definition}[Nullclines]
	The $v$-nullcline is the set \[V_{null}\defi\{(v,w)\in\mathbb R^2\st F(v)-w+I=0\}\] and the $w$-nullcline is \[W_{null}\defi\{(v,w)\in\mathbb R^2\st v-w=0\}.\] We note $\Vn^\pm(w,I)$ the v component of the $v$-nullcline which is on the left / right of the minimum of $F$ (which exists thanks to assumption~\ref{hyp.F}).
\end{definition}

We now define the separatrix which is a particular solution of \eqref{eq:micro-unique} which partitions the state space in two domains. It allows to ``remove'' the left part of the state space $\mathbb R^2$ from the analysis without altering the dynamics.

\begin{definition}[Separatrix]
	\label{def:separatrix}
	A separatrix $x_{sep}^I=(\vsep^I,\wsep^I)$ is a curve defined as follows. Consider a solution $(\tilde v,\tilde w)$ of \eqref{eq:micro-unique} \textcolor{black}{for $\I=0$} below the $v$-nullcline such that $\lim_{t\to t_{-\infty}}\tilde v(t)=-\infty$ and which intersects the $w$-nullcline at a point $(w^*(I),w^*(I))$ such that
	\[
		w^*(I)<\min F+I.
	\]
	 Assuming that $(w^*(I),w^*(I))$ is the leftmost intersection point with the $w$-nullcline for $t=t^*$, we define $\vsep^I(t) = w^*(I)$ and $\wsep^I(t) = w^*(I)$  for $t\geq t^*$ and $( \vsep^I(t),  \wsep^I(t))=(\tilde v(t), \tilde w(t))$ otherwise.
\end{definition}

We note that a separatrix is a $C^1$ curve and no trajectory can cross a separatrix from above. Thus, the separatrix produces a partition of the state space because the part above the separatrix is flow invariant (see  \ref{fig:setA}). By choosing the separatrix low enough, we can consider any initial condition and we thus restrict the study to the part $\CP^I$ above a given separatrix, see gray area in figure~\ref{fig:setA}.

\subsubsection{Partition}
Let us define the reset line

\begin{definition}[Reset line]
	It is the set \[\CR\defi\{x\in \R^2\st x_1=\vr\}.\]
\end{definition}

\begin{figure}[h!]
	\centering
	\includegraphics[width=0.6\textwidth]{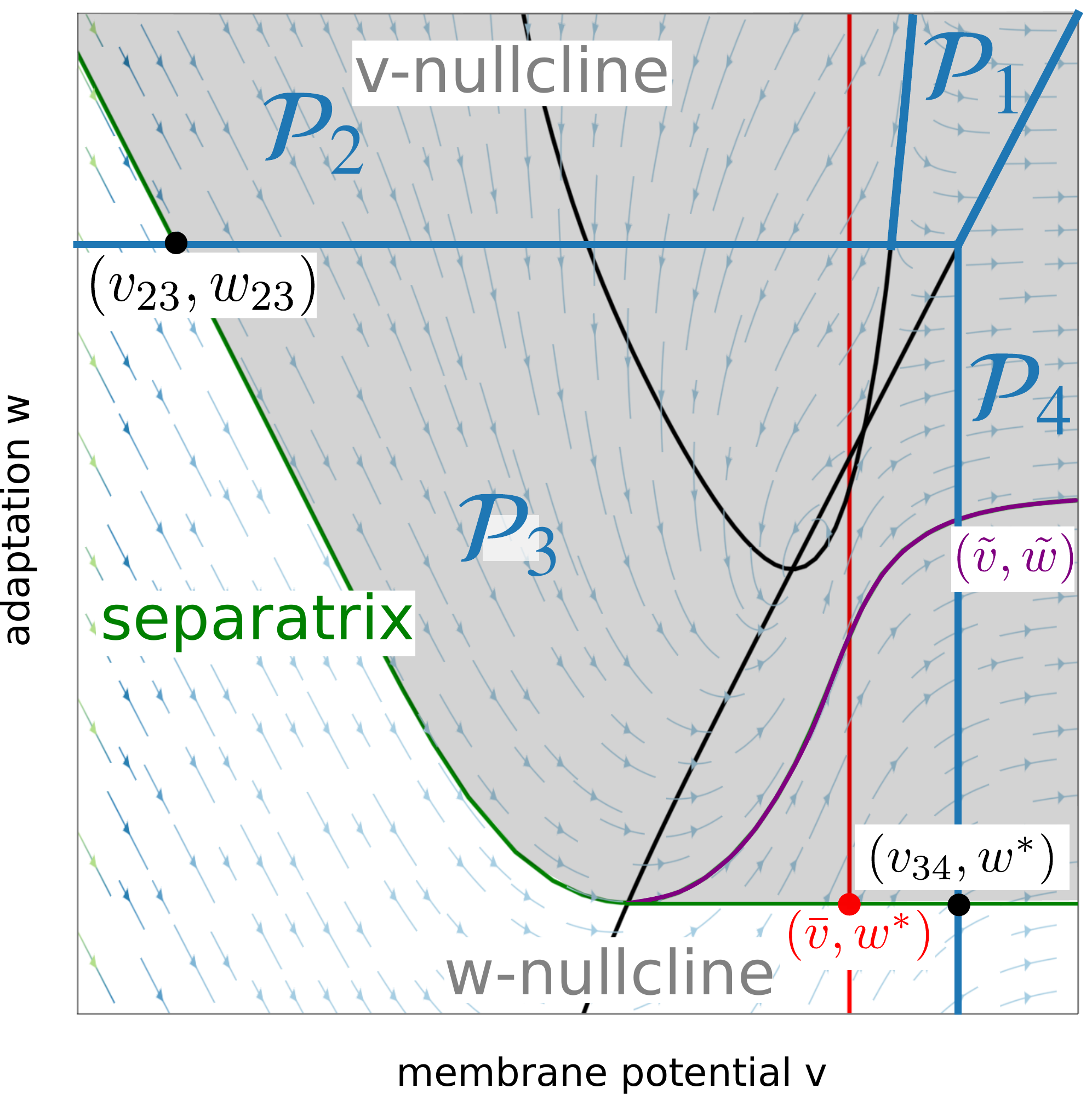}
	\caption{Plot of the vector field, the nullclines and a typical partition $\CP$ (gray region) in the case $\I=0$. The red line indicates the reset line $\{v=\vr\}$.  See text for explanations.}
	\label{fig:setA}
\end{figure}

We are now in position to provide a partition $\CP^I$ of the phase space of \eqref{eq:micro-unique} above a given separarix. This partition is built for $\I=0$. We choose a separatrix whose existence is guaranteed by proposition~\ref{prop:separatrix}. We write $\CP^I\defi \CP^I_{1}\cup\CP^I_{2}\cup\CP^I_{3}\cup\CP^I_{4}$, as in figure~\ref{fig:setA}. We first chose a horizontal boundary between  $\CP^I_{2}$ and $\CP^I_{3}$ which lies strictly above any equilibrium point or strictly above the minimum of the $v$-nullcline in case no equilibrium exists ; it is contained in $\{w=w_{23}^I\}$  and $v_{23}^I$ is the point on the separatrix with ordinate $w_{23}^I$. Additionally, we require that the reset line $\CR$ intersects the line $\{w=w_{23}^I\}$ strictly above the $v$-nullcline. The boundary between $\CP^I_{3}$ and $\CP^I_{4}$ is a vertical half-line extending up to the $w$-nullcline (which does not meet the $v$-nullcline); $v_{34}^I$ is the value corresponding to this boundary. Finally, the boundary between $\CP^I_{1}$ and $\CP^I_2$ is given by the right part of the $v$-nullcline. Additionally:
\begin{itemize}
	\item $\CP_1^I$ contains its boundary with $\CP_2^I$ but not the one with $\CP_3^I,\CP_4^I$.
	\item $\CP_2^I$ does not contain its boundary with $\CP^I_3$ and $\CP^I_1$.
	\item $\CP_3^I$ contains the corresponding part of the horizontal line $\{w=w_{23}^I\}$ but not $\{v=v_{34}^I\}$.
	\item $\CP_4^I$ contains the corresponding part of the line $\{v=v_{34}^I\}$.
\end{itemize}

\begin{remark}
	Note that $w_{23}^I=v_{34}^I$ by definition.
\end{remark}

\begin{remark}\label{rmq:CP}
	Note that given $w_{23}^I$ (resp. $v_{34}^I$), there is at most a single partition $\CP^I$ with associated boundary between $\CP_2^I$ and $\CP_3^I$ (resp. between $\CP_3^I$ and $\CP_4^I$).
\end{remark}

\begin{definition}[Partition]
	$(x_{sep}^I, w_{23}^I)$ uniquely defines a partition $\CP^I$. As such, we write
	\[\CP^I\defi (x_{sep}^I, w_{23}^I).\]
\end{definition}

We have some freedom in the choice of $w_{23}^I$ and the separatrix $x_{sep}^I$ which can be chosen respectively larger and lower.
The processes $W^{nl}$ and $W^{s,\mu_0,\I}$ live in the space
\begin{equation}\label{def:Ew}
\Ew^I \defi [w^*(I),\infty).
\end{equation}
It is also convenient to define:
\begin{equation}\label{def:RP}
	\CR_{\CP}^I \defi  \CP^I\cap\CR.
\end{equation}
$\Ew^I$ and $\CR_{\CP}^I$ are intrinsically bound to a partition hence to a separatrix.

\begin{remark}
	When the meaning is clear, we drop the $I$ dependency and write $\CP, \Ew,\CR_{\CP},\cdots$
\end{remark}

\subsection{Stochastic part}
The \emph{rate function} $\lambda(x)=\lambda(x_1)$ only depends on the first component of $x$ (we will use both notations $\lambda(x)$ and $\lambda(v)$) and may be subjected to the following conditions:

\begin{asumption}
	\label{as:ratepos-increasing}
	The rate function $\lambda$ is measurable, positive, non decreasing and such that:
	\[
	\int_{v_1}^{+\infty}\frac{\lambda}{F}=+\infty
	\]
	for any $v_1$ such that $F(v)>0$ for all $v\geq v_{1}$ (which exists according to assumption~\ref{hyp.F}--$(ii)$).
\end{asumption}
The monotony assumption is used quite intensively throughout the paper.

\begin{asumption}
	\label{as:C}
	$\forall \I\geq 0$, $\sup_{v\geq v_{1}} \lambda(v)\exp\left(-\int_{v_{1}}^v\frac{\lambda}{F+\I}\right)<\infty$ for any $v_1$ such that $F(v)>0$ for all $v\geq v_{1}$.
\end{asumption}

\begin{asumption}
	\label{as:C0}
	$\forall \I\geq 0$, $\lim\limits_{v\to\infty} \lambda(v)\exp\left(-\int_{v_{1}}^v\frac{\lambda}{F+\I}\right)=0$ for any $v_1$ such that $F(v)>0$ for all $v\geq v_{1}$.
\end{asumption}

\begin{asumption}
	\label{as:C2}
	$\forall \I\geq 0$, $\sup_{v\geq v_{1}} \lambda(v)^2\exp\left(-\int_{v_{1}}^v\frac{\lambda}{F+\I}\right)<\infty$ for any $v_1$ such that $F(v)>0$ for all $v\geq v_{1}$.
\end{asumption}

Note that assumption~\ref{as:C2} implies assumption~\ref{as:C0} which implies assumption~\ref{as:C}.

\begin{asumption}
	\label{as:r-lambda-F}
$\lambda$ belongs to $C^1(\R,\R_+^*)$ and for all $\alpha\geq 0$, there is $C_{\lambda,\alpha}>0$ such that \[\forall v\in\mathbb R,\quad \lambda'(v)(F(v)+\alpha)-\lambda^2(v)\leq C_{\lambda,\alpha}.\]
\end{asumption}

\noindent
 It is convenient to define the following function to describe the jumps of the individual neurons
\begin{equation}
	\label{eq:Delta}
	\Delta(v,w)\defi(\vr, w+\wb)
\end{equation}
onto the reset line $\CR$.

We define the \emph{survival function} (written $\Lambda^\I(x,t)$ in the autonomous case):
\begin{align}
\label{eq.survival.function}
  \Lambda^\I(s,x,t)
  \defi
  \int_s^{t\wedge\te{s,x,\I}} \lambda(\Phi_s^u(x))\,\rmd u.
\end{align}
Let $T_{0}=s$, the solution $X$ to \eqref{eq:EDSL} starting from $X_{0}$ follows the flow $\Phi_s^{t}(X_{T_{0}})$ until a time $T_{1}$ at which it jumps to $X_{T_{1}^+}=(\vr,X_{{T_{1}}^-}+\wb)=\Delta\Phi_s^{T_{1}}(X_{T_{0}})$ and so on.
The inter-jumps random variables 
\[S_{n}\defi T_{n}-T_{n-1}\]
have distribution
\begin{align}
\label{eq.Lambda}
  \P(S_{n}>u|X_{T_{n-1}}=x, T_{n-1}=s)
  =
  \exp\left(-\Lambda^\I(s,x,s+u)\right),\ u\geq 0
\end{align}
which, for each given $x$, have the density\footnote{with abuse of notations as the flow is not defined beyond $\tee$} \textit{w.r.t.} the Lebesgue measure
\begin{align}
\label{def.p}
p^\I(s,x,t)
  \defi
  \lambda(\Phi^{t}_s(x))\,e^{-\Lambda^\I(s,x,t)}\,\indic_{[s,\te{s,x})}(t)
\end{align}
under the assumption~\ref{as:ratepos-increasing} which ensures that $T_1<\tee$ \textit{a.s.}. 

We write $(\bfw_{n},S_{n})_n$ the enclosed Markov chain where 
\begin{equation}\label{eq:enclosed}
	(\vr,\bfw_n) \defi X^{s,x,\I}_{T_n^+}.
\end{equation}

We consider the canonical filtration $(\mathcal F_t^s)_{t\geq s}$ associated
to the Poisson measure $\mathbf N$ and to the initial condition $X_s^{s,\mu_0,\I}$, that is the completion of
\[
\sigma\left\{X_s^{s,\mu_0,\I}, \mathbf{N}([s, r] \times A): s \leq r \leq t, A \in \mathcal{B}\left(\mathbb{R}_{+}\right)\right\}.
\]

\begin{definition}\label{def:sol-sde-inhomo}
	Let $s\geq 0$ and $\I\in C(\R, \R_+)$.
	\begin{itemize}
		\item A process $(X_t^{s,\mu_0,\I})_{t\geq s}$ is said to be a solution of the non-homogeneous linear equation \eqref{eq:EDSL} with current $\I$ if the law of $X_s^{s,\mu_0,\I}$ is $\mu_0$, $(X_t^{s,\mu_0,\I})_{t\geq s}$ is $(\mathcal F_t^s)_{t\geq s}$-adapted, càdlàg, $t\to\lambda(X^{s,\mu_0,\I}_t)$ is locally integrable \textit{a.s.} and \eqref{eq:EDSL} holds \textit{a.s.}
		\item A $(\mathcal F^0_t)_{t\geq 0}$-adapted càdlàg process $(X^{nl}_t)_{t\geq 0}$ is said to solve the non-linear SDE \eqref{eq:MKV} if
		$t\to\E\lambda(X^{nl}_t)$ is locally integrable, if $(X^{nl}_t)_{t\geq -D}$ is a solution of \eqref{eq:EDSL} with $s=0,X^{nl}_t =X_0$ for $t\in[-D,0]$ and for all $t\geq 0$, $\I(t) = J\E\lambda(X_{0\wedge	t-D})$.
	\end{itemize}
\end{definition}

\section{Main results}\label{section:main-results}

We state the theorem of existence / uniqueness of a solution to \eqref{eq:MKV}. Note that the same results hold for \eqref{eq:EDSL} under the same assumptions.

\begin{theorem}\label{th-existence-MKV-delay}
	Grant assumptions~\ref{hyp.F},\ref{as:ratepos-increasing},\ref{as:C2},\ref{as:r-lambda-F} and consider a partition $\CP^I$. Assume that $X_0$, independent of $\mathbf N$, has law $\mu_0$ such that $\mu_0(\lambda)<\infty$ with $supp(\mu_0)\subset\CP^I$.
	Let us further assume that $D>0$. Then, there exists a unique $(\mathcal F_t^0)_{t\geq 0}$-adapted càdlàg process $X^{nl} = (V^{nl},W^{nl})$ solution to \eqref{eq:MKV}. In addition, $t\to\mathbb{E}\lambda(V_{t}^{nl})$ is continuous.
\end{theorem}

\noindent One important question remains. Is \eqref{eq:MKV} well posed for $D=0$?

\

The existence of stationary distributions for \eqref{eq:EDSL} in the case $\I=0$ is based on the study of the enclosed Markov chain $(\bfw_{n},S_{n})_n$. We note $\bfG$ its transition kernel and $\bfG^*$ its formal adjoint.

\begin{theorem}\label{th-mai-results-muW}
	Grant assumptions~\ref{hyp.F},~\ref{as:ratepos-increasing}. Assume that $\lambda$ is continuous and that 
	\begin{equation}\label{eq:Vnullhyp}
	\int^\infty_{w_{23}}\frac{\lambda(\Vn^-(w,I))}{w-\Vn^-(w,I)}dw<\infty.
	\end{equation}
for a partition $\CP$. Then, the following holds:
\begin{enumerate}
	\item $\bfG^{2k_D}$ satisfies a (global) Doeblin condition on $\Ew\times\R_+$ for some $k_D\in\mathbb N^*$ and $(\bfw_{n},S_{n})_n$ has a unique invariant distribution $\mu^{\mathbf w,S}$.
	\item There are $\alpha \in(0,1)$ and $C>0$ such that if $(w_0,S_0)$ has law $\nu_0$ with  $supp(\nu_0)\subset \Ew\times\R_+$, then:
	\[
	\norm{\bfG^{*n}\nu_0-\mu^{\mathbf w,S}}_{TV}\leq C\alpha^n\norm{\nu_0-\mu^{\mathbf w,S}}_{TV}.
	\]
	\item Denote by $\muW$ the invariant distribution of the Markov chain $(\bfw_{n})_n$. There is $r>0$ such that $\muW(e^{r\cdot})<\infty$ whence $\muW(w,\infty)\underset{w\to\infty}{=}O(e^{-rw})$.
	\item One has \[\E_{\muW}(T_1)<\infty.\]
		\item $\muW$  is not compactly supported, has density $\pW$ with respect to the Lebesgue measure when $I\in\Ireg$ which is continuous on $[w_{23}+\wb,\infty)$, is positive on $[w_{23},\infty)$ a.e. and has asymptotic behavior
	\[ \pW(u) =o\left(\frac1u\right).\]	
\end{enumerate}
\end{theorem}

Actually, the results of section~\ref{section:DI} are more general than the ones of theorem~\ref{th-mai-results-muW} in that they are locally uniform in the parameter $I$ in \eqref{eq:micro-unique}.

\begin{theorem}
	\label{thm:existence-uniq-inv-c0}
	Grant assumptions of theorem~\ref{th-existence-MKV-delay}and consider a partition $\CP^I$. Assume that $\lambda$ is continuous and that \eqref{eq:Vnullhyp} is satisfied. 
	Then, the solution $X$ of \eqref{eq:EDSL} is ergodic on $\CP^I$ with unique invariant distribution
	\begin{equation}\label{eq:muinv}
	\mu_I^{inv}(h)
	=	\frac{1}{\E_{\muW_I}(T_1)}
	\E_{\muW_I}\int_0^{T_1} h\circ\Phi^s\,ds
	\end{equation}
	and $supp(	\mu_I^{inv})\subset\CP^I\setminus\CP^I_1$.
\end{theorem}
We conjecture that $X$ is Harris recurrent on some subdomain of the partition $\CP^I$, this would be quite easy to prove if $\bfG$ satisfied a global Doeblin estimate. This question is linked to the general problem of equivalence between the stability properties of the enclosed and continuous time processes \cite{costa_stationary_1990,benaim_qualitative_2015}.
We now look at the invariant distributions of the SDE \eqref{eq:MKV}. 

\begin{theorem}\label{thm-nlDI}
	Grant assumptions of theorem~\ref{th-mai-results-muW}.
	For $J\geq 0$ small enough, there is a (nonlinear) invariant distribution to \eqref{eq:MKV} when $I\in \Iempty$.
\end{theorem}

This result is not satisfying in at least two ways. First, we have to assume that $I\in\Iempty$ in order to prove the technical result that $I\to\muW_I$ is continuous (see proposition~\ref{prop:di-c0}). However, when $J=0$, there is a unique invariant distribution even when $I\notin\Iempty$ by theorem~\ref{th-mai-results-muW} and we'd expect this to hold for $J$ small. Then, we have not answered the question of what happens for large $J$. This is linked to the asymptotic behavior of $\mathbb R\ni\I\to \E_{\muW_\I}(T_1)$, or basically to the boundedness of $\I\to \I\cdot\E_{\muW_\I}(T_1) $.

\section{Preliminary results on the deterministic flow}\label{section:ode}

We derive some results about the flow of the ODE  \eqref{eq:micro-unique} which are used repeatedly throughout the paper.
The basic assumption concerning the vector field is assumption~\ref{hyp.F}.
Strict convexity implies that the vector field has at most two equilibrium points when $\I$ is constant. Negative derivative $\limsup\limits_{-\infty}F'$ provides the existence of a separatrix.
Assumption~\ref{hyp.F}-$(ii)$  gives finite time explosion.

We now prove the existence of $\CP^I$.
\begin{proposition}
\label{prop:separatrix}
Grant assumption~\ref{hyp.F}-$(i)$. Then, there is a separatrix $x^I_{sep}$ (see definition~\ref{def:separatrix}) such that 
\[\lim\limits_{-\infty} w^I_{sep}=+\infty.\] 
As a consequence, $\CP^I$ exists.
\end{proposition}
\begin{proof}
By lemma~\ref{lemma:separatrix}, there is a set $B_{\alpha_-,x_n}$ with $x_n\defi(v_-,v_-)$, which is invariant by the backward dynamics of \eqref{eq:micro-unique}. 
Any point in $B_{\alpha_-,x_n}$ yields a separatrix using the backward flow.
To conclude on the existence, we observe that the $w$-nullcline intersects $B_{\alpha_-,x_n}$ and that $B_{\alpha_-,x_n}$ is located strictly below the $v$-nullcline.

Let us write $\tilde x(t)$ the solution for backward flow. In $B_{\alpha_-,x_n}$, $\tilde v$ decreases, $\tilde w$ increases and $\tilde v$ is unbounded. $\tilde x(t_0)$ is strictly above the $w$-nullcline for some $t_0$. It yields for $t\geq t_0$ 
\[\tilde w(t)\geq v(t_0)+e^{t-t_0}(\tilde w(t_0)-\tilde v(t_0))\]
which is unbounded.
\end{proof}

We define $V_x$, away from the $w$-nullcline, the $v$-component of the flow, as the solution of
\begin{equation}\label{eq:Vx}
\left\{
\begin{aligned}
\frac{dV_x(w)}{dw}&= \frac{F(V_x(w))-w+I}{V_x(w)-w},\\
V_x(x_2) &= x_1.
\end{aligned}
\right.
\end{equation}
It follows that $(v_x(t),w_x(t)) = (V_x(w_x(t)),w_x(t))$.
Similarly,  away from the $v$-nullcline, $W_x$ is the $w$-component of the flow, solution of
\begin{equation}\label{eq:Wx}
\left\{
\begin{aligned}
\frac{dW_x(v)}{dv}&= \frac{v-W_x(v)}{F(v)-W_x(v)+I},\\
W_x(x_1) &= x_2.
\end{aligned}
\right.
\end{equation}
It follows that $(v_x(t),w_x(t)) = (v_x(t),W_x(v_x(t)))$.

\begin{theorem}
\label{th:jtb1}
Grant assumption~\ref{hyp.F}. Let $\I\in C^0(\R,\R_+)$ and consider a partition $\CP^I = (x_{sep}^I, w_{23}^I)$, then the solution of \eqref{eq:micro-unique} satisfy the following properties.
\begin{enumerate}
\item The solutions stay in $\CP^I_4$ and blow up in finite time. The adaptation variable $w$ is finite at the time of explosion.
\item The solutions from $\CP^I_1$ hit $\CP^I_3$ or $\CP^I_4$ in finite time.
\item The solutions stay above the separatrix $x_{sep}^I$.
\item The solutions from $\CP^I_2$ hit $\CP^I_3$ or $\CP^I_1$  in finite time.
\item The solutions from $\CP^I_3$ either hit $\CP^I_4$ (in finite time) or stay trapped in $\CP^I_3$.
\end{enumerate}
In particular, the solutions are either bounded (and trapped in $\CP^I_3$) or explode in finite time (through $\CP^I_4$).
\end{theorem}
\begin{proof}
	When $\I\in C^0(\R,\R_+)$, the existence / uniqueness of a solution to \eqref{eq:micro-unique} is classical. We note that for $\I\geq 0$, the $v$-nullcline is translated upwards at each time $t$ compared to the case $\I=0$.
	\begin{enumerate}
		\item Hence, the $v$-component of the vector field is positive in $\CP^I_4$. It implies that the solutions cannot escape $\CP^I_4$. 
		In $\CP^I_4$: $w_x(s,t)\leq v_x(s,t)$ and thus $\dot v_x(s,t)\geq F(v_x(s,t)) - v_x(s,t)+I$. By Gronwall's lemma for scalar ODEs, the solution explodes in finite time.
		Then, by a nonlinear time change for  $x\in\CP_4^I$ and $s$ real, we find that $W_{s,x}(v_{s,x}(t)) = w_{s,x}(t)$ for $t\geq s$ where
		\[
		\forall v\geq x_1,\quad \frac{d}{dv}W_{s,x}(v) = \frac{v-W_{s,x}(v)}{F(v)-W_{s,x}(v)+I+\I(\tau_s(v))},\quad W_{s,x}(x_1) = x_2,
		\]
		with $\tau_s(v) \defi \inf\{t\geq s,\ v_x(s,t)=v\}$.
		This implies that \[\frac{d}{dv}W_{s,x}(v)\leq \frac{v-x_2}{F(v)-v+I}.\]
	 Assumption~\ref{hyp.F}~(ii) implies that $W_{s,x}$ is upper bounded. This concludes the item.
		\item Given the vector field, the only way to escape $\CP^I_1$ is by $\CP^I_3$ or $\CP^I_4$. As long as the solution stays in $\CP^I_1$, we have $w_x(s,t)\leq w_x(s,s) = x_2$. It implies that $\dot v_x(s,t)\geq F(v_x(s,t)) -x_2 +I$ and by Gronwall's lemma, that the solution crosses the $w$-nullcline in finite time. This proves the point.
		\item Recall that the separatrix is defined for $\I=0$. First, the solution from $x\in\CP^I$ and the separatrix cannot intersect (from above) below the $w$-nullcline because the vector field is upward. We now consider the first intersection point $x_0\defi(v_0,w_0)$, if it exists, for some $w_0 > w^*$ and $s$. From \eqref{eq:Wx} for $v\geq v_0$,
		\[
		0\geq\frac{d}{dv}W_{s,x_0}(v) = \frac{v-W_{s,x_0}(v)}{F(v)-W_{s,x_0}(v)+I+\I(\tau_{s}(v))}\geq \frac{v-W_{s,x_0}(v)}{F(v)-W_{s,x_0}(v)+I}
		\]
		where we used that $v-W_{s,x_0}(v)< 0 < F(v)-W_{s,x_0}(v)+I$.
		Let $W^I_{sep}$ be the solution of \eqref{eq:Wx} corresponding to the separatrix $x_{sep}^I$.
		Gronwall's lemma for scalar ODEs implies that $W^I_{sep}(v)\leq W_{s,x_0}(v)$ for $v\geq v_0$ and the trajectory stays above the separatrix.
		\item Since $\CP^I_2$ is above the $w$-nullcline, there is a constant $C_{x}>0$ so that $\dot w_x(s,t)\leq -C_{x}$ in $\CP^I_2$. Hence, the solutions leave $\CP^I_2$ in finite time while staying above the separatrix.
		\item If a solution escapes $\CP^I_3$, it can only occur at the boundary with $\CP^I_4$ because the solution stays above the separatrix.
	\end{enumerate}
\end{proof}

Next, we provide a simple lemma concerning the partitions. Basically, we can chose a ``single'' partition that is invariant by the stochastic flow for all $I$ in a bounded interval.

\begin{lemma}\label{lemm_partition}
	Grant assumption~\ref{hyp.F}.
	Given an interval $[\underline{I}, \bar I]$, there is a family of partitions $\CPf{\underline{I}}{\bar I}$ for which $w_{23}$, $v_{34}$ and $\Ew=[w^*,\infty)$ are independent of $I$. 
\end{lemma}
\begin{proof}
	We consider a partition $\CP^{\bar I}$ for the current $\bar I$: its $w_{23}$ can be used for all the partitions for $I\in[\underline{I}, \bar I]$. From the proof of lemma~\ref{lemma:separatrix}, we can use the same $B_{\alpha_-,x_n}$ for all $I\in [\underline{I}, \bar I]$ by choosing $v_-$ small enough where $x_n\defi(v_-,v_-)$. Indeed, the condition on $v^*$ in the quoted proof is independent of $I$. We can thus choose the same point $(w^*,w^*)$, intersection of the separatrix with $w$-nullcline, for all $I$: this provides a single $\Ew$ for all $I$.
\end{proof}

\begin{remark}
	Note that the separatrix for each $\CP^I$ are different in the previous lemma: $\CP^I\defi (x_{sep}^I, w_{23})$.
\end{remark}

\begin{corollary}\label{coro:wbound}
	Grant assumption~\ref{hyp.F}. Let $\underline I\leq \bar I$ and assume that $\I=0$.
	Consider a family of partitions $(\CP^I)_{I\in [\underline{I}, \bar I]}$ from lemma~\ref{lemm_partition}. Then, there is $C>0$ such that
	\[\forall w_0\in\Ew,\ \forall t\geq 0, \ \forall I\in[\underline I, \bar I],\quad w^I_{(\vr,w_0)}(t)\leq C+w_0.\]
\end{corollary}
\begin{proof}
We have the symbolic dynamics $\CP_2^I\cap\CR\to\CP^I_3\stackrel{(possibly)}{\to}\CP^I_4$ because $\CP^I_2$ cannot hit $\CP_1^I$ from the reset line $\CR$ when $\I=0$. In $\CP_2^I$, $w^I$ decreases and thus $w^I(t)\leq w^I(0)$ for all time $t\geq 0$. In $\CP^I_3$, the trajectory is either bounded and $w^I(t)\leq w_{23}$ or it is unbounded and explodes through $\CP^I_4$. In $\CP^I_4$, the proof of theorem~\ref{th:jtb1}~(i) shows that 
\[
w^I_x(t)\leq w^I(0)+\sup\limits_{v\geq v_{34}}\int_{v_{34}}^v \frac{u-w^*}{F(u)-u+\underline I}du\leq w_{23}+\int_{v_{34}}^\infty \frac{u-w^*}{F(u)-u+\underline I}du<\infty.
\]
This provides the bound stated in the lemma.
\end{proof}

\section{Study of the linearised process $X^{s,\mu_0,\I}$ }\label{section:linear}
In this section, we analyze \eqref{eq:EDSL} as a proxy for \eqref{eq:MKV}. We first detail sufficient conditions for existence, uniqueness and regularity (almost sure finite number of jumps in bounded time intervals) of \eqref{eq:EDSL}. We then study the rate function $\E\lambda(X_t^{s,\delta_x,\I})$ and give sufficient conditions for it being finite and continuous in time. This will then be used in the proof of theorem~\ref{th-existence-MKV-delay} in section~\ref{section:NL-DDE}.

We first give a sufficient condition for the jumps to occur before the deterministic explosion time almost surely.
\begin{lemma}
\label{lemma.Lambda}
Grant assumptions~\ref{hyp.F},\ref{as:ratepos-increasing} and assume that $\I\in C^0(\R,\R_+)$. Let us consider a partition $\CP$. Then for all real $s$ and $x\in\CP$:
\begin{enumerate}
	\item $T^{s,x,\I}_1>0$ a.s. and $T^{s,x,\I}_1<\te{s,x,\I}$ \textit{a.s.}.
	\item $\inf\{t\geq s\st\P(T^{s,x,\I}_1\geq t)=0 \} = \te{s,x,\I}$.
	\item The density of $\mathcal L(T^{s,x,\I}_1)$ is given by \eqref{def.p}.
\end{enumerate}
\end{lemma}
\begin{proof}
	\
	
\textbf{Point (i).}
We recall that a partition $\CP$ exists thanks to proposition~\ref{prop:separatrix} with assumption~\ref{hyp.F}.
We consider an initial condition $x\in\CP$. $T^{s,x,\I}_1>0$ a.s. is direct by \eqref{eq.Lambda}. To prove the point, we show that $\int_{s}^{\te{s,x}}\lambda(v_x(s,u))du=\infty$. We distinguish two cases. First, if $t\to v_x(s,t)$ is bounded then it is lower bounded by  some constant $\underline{v}$. In this case, the solution from $x$ is defined on  $[s,\infty)$ and $\int_s^{\te{s,x}}\lambda(v_x(s,u))du\geq \int_s^\infty \lambda(\underline{v})=\infty$ by assumption~\ref{as:ratepos-increasing} which provides the claim.

If $t\to v_x(s,t)$ is unbounded,  then from theorem~\ref{th:jtb1}, the solution from $x$ explodes in finite time by passing in $\CP_4$ at hitting time $\tau_4^{s,x,\I}<\infty$. In $\CP_4$,  a time change gives
\[
\int_{\tau_4^{s,x,\I}}^{\te{s,x}}\lambda(v_s^u(x))du=\int_{v_{34}}^\infty\frac{\lambda(v)}{F(v)-W_{s,x}(v)+I+\I(\tau_s(v))}dv\geq \int_{v_{34}}^\infty\frac{\lambda( v)}{F( v)- w_x(s,\tau_4)+I+\sup\limits_{t\in[s,t_\infty^{s,x,\I}]}\I(t)}d v.
\]
This last term is $+\infty$ according to assumption~\ref{as:ratepos-increasing}. The point follows.

\textbf{Point (ii).}
Next we consider, $\eta(s,x)\defi \inf\{t\geq s\st\P(T^{s,x,\I}_1\geq t)=0 \}$. The previous point implies that $\eta(s,x)\leq \te{s,x,\I}$. We further have $\eta(s,x)= \te{s,x,\I}$ because $\Lambda^\I(s,x,t)<\infty$ if $s<t<\te{s,x,\I}$. 

\textbf{Point (iii).} Direct consequence of the expression of $\P(T^{s,x,\I}_1\geq t)$ and (i).
\end{proof}

\subsection{Regularity}

In this section, we prove that the solutions of the SDE \eqref{eq:EDSL} are defined for all time. We recall \cite{jacod_jumping_1996} that $X^{s,x,\I}$, solution of \eqref{eq:EDSL} in some partition $\CP$, is regular if the jump times $T_n^{s,x,\I}$ satisfy $\lim_nT_n^{s,x,\I} = \infty$ $P_x$-\textit{a.s.} for all $x\in\CP$ with the notations of \cite{blumenthal_markov_1968}.

\begin{proposition}\label{prop:regularity}
Grant assumptions~\ref{hyp.F},\ref{as:ratepos-increasing}, assume that $\I\in\mathrm C^0(\R,\R_+)$ is bounded and consider a partition $\CP^I$. For $x\in\CP^I$, any solution $X^{s,x,\I}$ of \eqref{eq:EDSL} is regular.
\end{proposition}
\begin{proof}
Let $\tau_4^{s,x,\I}$ be the deterministic hitting time of $\{v=v_{34}+1\}$ from $x\in\CP^I$ at initial time $s$.
Let us write $\CP^I = (x_{sep}^I,w_{23}^I)$. We consider another $w_{23}$ ``valid'' for the current $I+2\norm{\I}_\infty$ and define $\tilde \CP^I \defi (x_{sep}^I,w_{23})$. Since, $\CP^I = \tilde\CP^I$, we work with $\tilde\CP^I$.

\textbf{First step.} We prove that $\P(T_1^{s,x,\I}\geq s+\delta)\geq \alpha$ for all $s\in\R$ and $x\in (\tilde\CP_2\cap\CR)\cup\tilde\CP_3$ for some $\delta,\alpha>0$. From the vector field, for all $t\in[s,\tau_4)$ and $x\in(\tilde\CP_2\cap\CR)\cup\tilde\CP^I_3$, we find $\lambda(v_x(s,t))\leq \lambda(v_{34}+1)$. Because $\I$ is bounded, we find that
\[
\delta_1\defi\inf\limits_{s\in\R,x\in\tilde\CP_3}\tau_4^{s,x,\I}>0.
\]
Indeed, the hitting times are lower bounded by 
\[
\tau_4^{s,x,\I}\geq \int_{v_{34}}^{v_{34}+1} \frac{1}{F(\underline v)-w^*(I)+I+\norm{k}_\infty}d\underline v>0
\] which estimates the time it takes to reach $\{v=v_{34}+1\}$ from $\{v=v_{34}\}$, see proof of lemma~\ref{lemma.Lambda}.
Next, we observe that the solutions leave $\tilde\CP_2\cap\CR$ through $\tilde\CP_3\cap\{v\leq \vr\}$. Hence, for $x\in \tilde\CP_2\cap\CR$ we find $\tau_4^{s,x,\I}\geq\delta_1$ and $\lambda(v_x(s,t))\leq \lambda(v_{34}+1)$ for $s\leq t\leq \delta_1$.
Hence, if we define $\delta\defi\min(1,\delta_1)$ to cope with the case $\delta_1=\infty$, we obtain that $\Lambda^\I(s,x,s+\delta)\leq \lambda(v_{34}+1)\delta$ and thus $\alpha\defi\exp\left(- \lambda(v_{34}+1)\delta\right)$.

\textbf{Second step.} 
We introduce the time homogeneous process  $\hat X_t^{(s,x),\I}\defi(X^{s,x,\I}_{t+s},t+s)$ with initial condition $(x,s)$ and its embedded Markov chain $(\hat X_n^{(s,x),\I}, S_n^{s,x,\I})$.
The existence of $X^{s,x,\I}$ is proved in the next proposition.
We denote by $\hat T_1^{(s,x),\I}$ the first jump time of $\hat X$. One has $\hat T_1^{(s,x),\I}+s = T_1^{s,x,\I}$.

\textbf{Last step.}
Without loss of generality, we restrict to the case $x\in\CR_{\CP}^I$ to prove that $\hat X$ hence $X$ is regular.
For $z>0$, we introduce the transition operator $\hat\bfG_z$ as follows. For $f\in \mathrm L^\infty(\CR_{\CP}^I\times \R,\R_+)$:
\[
\forall x\in\CR_{\CP}^I,\ \forall s\in\R,\quad \hat\bfG_zf(x,s)\defi \mathbb E\left(f(\hat X_{\hat T_1}^{(s,x),\I})e^{-z \hat T_1}\right).
\]
We find that $\hat\bfG_z\mathbf 1(x,s)\leq e^{-z\delta}+1-\alpha<1$ for $z$ large enough using the first step. From the proof of \cite{jacod_jumping_1996}~theorem~12, if $\hat X$ is not regular, then $(\hat\bfG_z)^n\mathbf 1(x,s)$ decreases to the bounded non negative function $g(x,s) \defi\mathbb E\left(e^{-z\hat T_\infty^{(s,x),\I}}\right)$ and there is $(x,s)\in\CR_{\CP}^I\times\R$ such that $g(x,s)>0$. But this can't be by the bound on $\hat\bfG_z1(x,s)$.
\end{proof}

The solutions of \eqref{eq:EDSL} are PDMPs which have been studied intensively in the literature but mainly for globally defined deterministic flows. In order to shorten our proofs, we use the very general setting of \cite{jacod_jumping_1996}.

\begin{proposition}\label{prop:JMP}
	Grant assumptions~\ref{hyp.F},\ref{as:ratepos-increasing}. Let $\I\in C^0(\R,\R_+)$ bounded and $x$ in a partition $\CP$. Then there is a path-wise unique solution $X^{s,x,\I}$ to \eqref{eq:EDSL} on $[s,\infty)$ in the sense of definition~\ref{def:sol-sde-inhomo}. Additionally, the solution is strong Markov with values in $\CP$.
\end{proposition}
\begin{proof}
	Let $s\in\R$ and $x\in\CP$.
	We give a direct proof by considering the jumps of $X^{s,x,\I}$ and by solving the equation between the jumps. Define by induction: $T_0\defi s$,
		\begin{align*}
	T_{1}&\defi\inf \left\{t \geq s: \int_{s}^{t} \int_{\mathbb{R}_{+}} \mathbf{1}_{\left\{z \leq \lambda\left(\Phi_{s}^{u}\left(X_s^{s,x,\I}\right)\right)\right\}} \mathbf{N}(d u, d z)>0\right\},\\
	X_t& \defi \Phi_{s}^t\left(X_s^{s,x,\I}\right),\quad t\in[s,T_1)
	\end{align*}
	and
	\begin{align*}
	\forall n \geq 0,\quad T_{n+1}&\defi\inf \left\{t \geq T_{n}: 
	\int_{T_{n}}^{t} \int_{\mathbb{R}_{+}} \mathbf{1}_{\left\{z \leq \lambda\left(\Phi_{ T_{n}}^u\left(\Delta X_{T_n^-}^{s,x,\I}\right)\right)\right\}} \mathbf{N}(d u, d z)>0\right\}\\
	X_t& \defi \Phi_{T_n}^t\left(\Delta X_{T_n}^{s,x,\I}\right),\quad t\in[T_n,T_{n+1}).
	\end{align*}
	Note that by lemma~\ref{lemma.Lambda} and equation \eqref{eq.Lambda}, $T_{n}<T_{n+1}<\infty$ \textit{a.s.} for $n\geq 0$ so that the sequence is well defined.	
	We can directly verify that $t\to X_t^{s,x,\I}$ is almost surely a solution of \eqref{eq:EDSL} which implies by proposition~\ref{prop:regularity} that $X^{s,x,\I}$ is regular, hence defined on $[s,\infty)$.
	
	Uniqueness of solutions to \eqref{eq:EDSL} follows from theorem~\ref{th:jtb1}: two solutions are almost surely identical before the first jump, implying that they must jump simultaneously. By induction on the number of jumps, we conclude that the two solutions are almost surely identical.
	
	The Markov property is a consequence of corollary 3 in \cite{jacod_jumping_1996}. In order to apply this corollary, we have to show two properties. First $\P(T^{s,x,\I}_1>s)=1$ which is true by \eqref{eq.Lambda}. Second, we consider $\eta(s,x)\defi \inf\{t\geq s\st\P(T^{s,x,\I}_1\geq t)=0 \}$. We have seen in lemma~\ref{lemma.Lambda} that $\eta(s,x)= \te{s,x,\I}$.  This implies that $\Phi_t^u(\Phi_s^t(x))=\Phi^{u}_s(x)$ holds for $s<t<u<\eta(s,x)$. Then $X$ is strong Markov from corollary 3 in \cite{jacod_jumping_1996} and the proof is complete.
\end{proof}

\subsection{Study of the jump rate}
We proved that the process $X^{s,x,\I}$ is defined on $[s,\infty)$. 
For any measurable and bounded $h$, the strong Markov property gives \cite{jacod_jumping_1996}:
\[
\mathbb E h\left(X^{s,x,\I}_t\right) = 	h(\Phi^{t}_s(x))\,e^{-\Lambda^\I(s,x,t)}\indic_{[s,\te{s,x})}(t)	 + \int_s^t p^\I(s,x,u)\ \mathbb E h\left(X_t^{u,\Delta\Phi_s^{u}(x),\I }\right)du.
\]
Hence, if the jump rate $r^\I(s,x,t)\defi \E\lambda(X_t^{s,x,\I})$ is finite, we expect it to satisfies the integral equation:
\begin{equation}\label{eq:ve}\tag{IE}
	r^\I(s,x,t) = p^\I(s,x,t) + \int_s^t p^\I(s,x,u)r^\I(u,\Delta\Phi_s^{u}(x),t)\,du,\quad s\leq t, \quad x\in \CP.
\end{equation}

This equation has a nice structure: when solved solely on the reset line $\CR_\CP$, it provides an explicit solution on the whole domain $\CP$. Hence, we start by solving the equation on $\CR_\CP$.

\begin{proposition}\label{prop:vel-R}
	Grant assumptions~\ref{hyp.F},\ref{as:ratepos-increasing},\ref{as:C}, consider $S<T<\infty$, $\I\in C^0([S,T],\R_+)$ and a partition $\CP$. Then, equation  \eqref{eq:ve} posed on the reset line $\CR_\CP$ has a unique non negative solution $\tilde r^\I$ in $\mathcal E_\CR\defi \mathrm L^{\infty}(J,\mathrm L^\infty(\CR_\CP))$ where $J\defi\{(s,t)\in\ [S,T]^2\st s\leq t\}$. 
	Additionally, under assumption~\ref{as:C0}, $\forall x\in \CR_\CP$ and $\forall s\in[S,T]$, $t\to\tilde r^\I(s,x,t)$ is continuous for $s\leq t\leq T$.
\end{proposition}
\begin{proof}
	We prove  that \eqref{eq:ve} has a unique solution for $x\in\CR_\CP$.
	We note that $p^\I$ belongs to $\mathcal E_\CR$ thanks to lemma~\ref{lemma:p_on_vr} and the three assumptions.
	Let us write \eqref{eq:ve} as $\tilde r= p^\I+\mathcal T_o(\tilde r)$. For $n\geq1$:
	\begin{multline*}
		\mathcal T^n_0(h)(s,x,t) = \int_s^tdu_1\int_{u_1}^tdu_2\cdots\int_{u_{n-1}}^tdu_n\\ 
		p^\I(s,x,u_1)\ 
		p^\I(u_1,\Delta\Phi_s^{u_1}(x),u_2)\cdots 
		p^\I(u_{n-1},\Delta\Phi_s^{u_{n-1}}\circ\cdots\circ\Delta\Phi_s^{u_{1}}(x),u_n)\\\times
		h(u_n,\Delta\Phi_s^{u_n}\circ\cdots\circ\Delta\Phi_s^{u_{1}}(x),t),
	\end{multline*}
	and we find:
	\[
	\mathcal T^n_0(h)(s,x,t)\leq \norm{h}_{\mathcal E_\CR}\norm{p}_{\mathcal E_\CR}^n\frac{(T-S)^n}{n!}<\infty.
	\]
	This implies that  $\mathcal T_0^{n_c}\in\mathcal L(\mathcal E_\CR)$ is contracting for some $n_c\geq1$ and that  \eqref{eq:ve} has a unique solution in $\mathcal E_\CR$ given by $\tilde r = \sum_{n\geq 0}\mathcal T_0^{n\cdot n_c}(p^\I)$. The solution is non negative because $p^\I$ is.
	Let us look at the continuity under assumption~\ref{as:C0}. Each term of the series is continuous in $t$ by lemma~\ref{lemma:p_C0} and bounded by $k^n\norm{p}_E^{n\cdot n_c}$ for some $k<1$. The discrete version of the theorem of dominated convergence implies that $t\to \tilde r^\I(s,x,t)$ is continuous for $t\geq s$.
\end{proof}
We can now state:

\begin{lemma}\label{lem:apriori-rate}
	Grant assumptions~\ref{hyp.F},\ref{as:ratepos-increasing},\ref{as:C}, consider $S<T<\infty$, $\I\in C^0([S,T],\R_+)$ and a partition $\CP$. Then the jump rate $r^\I$ is finite, is solution of \eqref{eq:ve} and there is a constant $M_{S,T}\geq0$ such that for all $x\in\CP$ and $\forall S\leq s\leq t\leq T$:
	\begin{equation}\label{eq:boundr}
	p^\I(s,x,t) \leq  r^\I(s,x,t)\leq  p^\I(s,x,t) + M_{S,T}.
	\end{equation}
	Additionally, under assumption~\ref{as:C0}, for $ s\in[S,T]$, $t\to r^\I(s,x,t)$ is continuous when $s\leq t\leq T$.	
\end{lemma}
\begin{proof}
We first show that $r(s,x,t)$ is finite. To this end, we consider an increasing sequence $\lambda_n\in\mathrm L^{\infty}(\R^2,\R_+)$, independent of the second variable, such that $\lambda_n\to\lambda$ a.e. The theorem of monotone convergence implies that the increasing sequence $r^n$ is such that $r^n(s,x,t)\defi\E\lambda_n(X_t^{s,\delta_x,\I})\to r^\I(s,x,t)$ a.e. from below. By \cite{jacod_jumping_1996} and the monotony of the sequence $\lambda_n$, we obtain
\begin{multline*}
	\forall s\leq t, \quad \forall x\in \CR_\CP,\quad r^n(s,x,t) =
\lambda_n(\Phi^{t}_s(x))\,e^{-\Lambda^\I(s,x,t)}\indic_{[s,\te{s,x})}(t)	 + \int_s^t p^\I(s,x,u)r^n(u,\Delta\Phi_s^{u}(x),t)du
	\\ \leq p^\I(s,x,t) + \int_s^t p^\I(s,x,u)r^n(u,\Delta\Phi_s^{u}(x),t)du.
\end{multline*}
As $r^n(s,\cdot, t)\in\mathcal E_\CR$, proposition~\ref{prop:vel-R} implies that $r^n(s,x,t)\leq \tilde r^\I(s,x,t)<\infty$ for $x\in\CR_\CP$ and then $r^\I(s,x,t)\leq \tilde r^\I(s,x,t)$. In particular $r^\I\in\mathcal E_\CR$ which implies, by taking the limit $n\to\infty$ in the integral equation and by using uniqueness of the solution to \eqref{eq:ve} on $\CR_\CP$ that
\[
	\forall x\in\CR_\CP,\	r^\I(s,x,t)= \tilde r^\I(s,x,t).
\]
This implies that $r^\I$ is continuous in time under assumption~\ref{as:C0} on $\CR_\CP$. 

Now, we observe that \eqref{eq:ve} allows to extend $\tilde r^\I$ to $\CP$.
Then, we obtain that 
\[
		\forall s\leq t, \quad \forall x\in \CP,\quad r^n(s,x,t) \leq p^\I(s,x,t) + \int_s^t p^\I(s,x,u)r^\I(u,\Delta\Phi_s^{u}(x),t)du\defi \tilde r^\I(s,x,t),
\]
which implies that  $r^\I(s,x,t)\leq \tilde r^\I(s,x,t)<\infty$ for $x\in\CP$. As above, we conclude that $r^\I(s,x,t)= \tilde r^\I(s,x,t)$ for $x\in\CP$ and that $r^\I$ is continuous in time for $x\in\CP$ under assumption~\ref{as:C0}. At this stage, we also have the inequality
\[
	p^\I(s,x,t) \leq  r^\I(s,x,t)\leq p^\I(s,x,t)+\norm{r^\I}_{\mathcal E_\CR}
\]
which concludes the proof. 
\end{proof}

This allows to solve \eqref{eq:EDSL} for $X_0\sim\delta_x$ for $x\in\CP$.
To solve \eqref{eq:EDSL} for $X_0\sim\mu_0$  fairly  general, we need to control $\E\lambda(X_t^{s,\mu_0,\I})$ which, by \eqref{eq:boundr}, amounts to controlling $p^\I$. We use an \textit{a-priori} estimate of the jump rate.
\begin{lemma}\label{lem:apriori-rate2}
	Grant assumptions~\ref{hyp.F},\ref{as:ratepos-increasing},\ref{as:C2},\ref{as:r-lambda-F}, consider $S<T<\infty$, $\I\in C^0([S,T],\R_+)$ and a partition $\CP$.  Then for all $x\in\CP$, there is $C_\lambda>0$ such that
	\[
	\forall S\leq s\leq t\leq T,\quad r^\I(s,x,t) \leq -\frac{C_\lambda}{\lambda(\vr)}+\left(\lambda(x_1)+\frac{C_\lambda}{\lambda(\vr)}\right) e^{\lambda(\vr)(t-s)}.
	\]
	As a consequence, $\E\lambda(X_t^{s,\mu_0,\I})$ is finite provided that $\mu_0(\lambda)<\infty$ and $supp(\mu_0)\subset\CP$. In this case, it is also continuous as function of $t$.
\end{lemma}
\begin{proof}
\textbf{Step 1.} We shall quickly show that $r_2^\I(s,x,t)\defi\E\lambda^2(X_t^{s,x,\I})$ is finite. It is straightforward, under assumptions~\ref{hyp.F},\ref{as:ratepos-increasing},\ref{as:C2}, to adapt lemma~\ref{lemma:p_on_vr} to
\[
p_2^\I(s,x,t)\defi\lambda^2(v^\I_x(s,t))e^{-\Lambda^\I(s,x,t)}\,\indic_{[s,\te{s,x,\I})}(t)
\]
and show that it is bounded on $\CR_\CP$. The contracting property  of $\mathcal T_0$ in the proof of proposition~\ref{prop:vel-R} provides a unique solution $r_2^\I$ on $\CR_\CP$ to the equation
\[
	r^\I_2(s,x,t) = p_2^\I(s,x,u)  + \int_s^t p^\I(s,x,u)r_2^\I(u,\Delta\Phi_s^{u}(x),t)\,du,\quad s\leq t, \quad x\in \CR_\CP.
\]
One can show that this solution is finite and equal to $\E\lambda^2(X_t^{s,x,\I})$ as in the previous lemma.
Using the proof of lemma~\ref{lem:apriori-rate}, we can then  extend $r_2^\I$ to $\CP$ with the bound
\[
p_2^\I(s,x,t) \leq  r_2^\I(s,x,t)\leq  p_2^\I(s,x,t) + \norm{r_2^\I}_{\mathcal E_\CR}<\infty.
\]
\textbf{Step 2.}	
We apply the Ito formula (theorem II-32 in \cite{ikeda_stochastic_1989}). It gives that for all $S\leq s\leq t\leq T$ and $x\in\CP$, almost surely:
\begin{multline*}
 \lambda(V_t^{s,x,\I}) - \lambda(x_1) =\int_s^t\lambda'(V_u^{s,x,\I})\left[F(V_u^{s,x,\I})-W_u^{s,x,\I}+\I(u)\right] - \mathbf{M}_t^{s,x,\I}
	\\\leq \int_s^t\lambda'(V_u^{s,x,\I})\left[F(V_u^{s,x,\I})-w^*+\sup\limits_{[S,T]}\I\right]-\mathbf{M}_t^{s,x,\I}
\end{multline*}
where 
\[
\mathbf{M}_t^{s,x,\I}\defi \int_{s}^{t} \int_{\mathbb{R}_{+}} \left(\lambda\left(V_u^{s,x,\I}\right)-\lambda(\vr)\right) \mathbf{1}_{\left\{z \leq \lambda\left(V_{u-}^{s,x,\I}\right)\right\}} \mathbf{N}(d u, d z).
\]
We used that $w^*\leq W_t^{s,x,\I}$ for $x\in\CP$ and that $\lambda$ is non decreasing for the second line.
From step 1, we have $\E\int_{s}^{t} \lambda^2\left(V_u^{s,x,\I}\right) d u<\infty$ which implies (see \cite{ikeda_stochastic_1989} page 62) that

\[
\E\mathbf{M}_t^{s,x,\I} = \E\int_{s}^{t} \lambda\left(V_u^{s,x,\I}\right)\left(\lambda\left(V_u^{s,x,\I}\right)-\lambda(\vr)\right)  d u
\]
and
\[
r^\I(s,x,t) - \lambda(x_1)	\stackrel{Ass.~\ref{as:r-lambda-F}}{\leq} \int_s^tC_\lambda +\lambda(\vr)\cdot r^\I(s,x,u)du.
\]
The estimate of the lemma is then a consequence of the Gronwall's lemma.

Let us prove the continuity in time $t$. In $\E\lambda(X_t^{s,\mu_0,\I}) = \int r^\I(s,x,t)\mu_0(dx)$, the integrand is continuous in time for $x\in\CP$. Additionally, 
\[
	\forall S\leq s\leq t\leq T,\quad r^\I(s,x,t) \leq -\frac{C_\lambda}{\lambda(\vr)}+\left(\lambda(x_1)+\frac{C_\lambda}{\lambda(\vr)}\right) e^{\lambda(\vr)(T-S)}.
\]
The continuity is thus a consequence of the Lebesgue's dominated convergence theorem.
\end{proof}

\begin{theorem}\label{th:Knu}
	Grant assumptions~\ref{hyp.F},\ref{as:ratepos-increasing},\ref{as:C2},\ref{as:r-lambda-F} and consider $\I\in C^0([S,T],\R_+)$ for $S<T<\infty$.
	Let us assume that $X_0$, independent of $\mathbf N$,  has law $\mu_0$ such that $\mu_0(\lambda)<\infty$ and $supp(\mu_0)\subset\CP$ for some partition $\CP$.
	Then there is a path-wise unique solution to \eqref{eq:EDSL} defined on $[s,\infty)$ in the sense of definition~\ref{def:sol-sde-inhomo}.
	Additionally, for all $s\in[S,T]$, $t\to\E\lambda\left(X_t^{s,\mu_0,\I}\right)$ is finite and continuous on $s\leq t\leq T$.
\end{theorem}
\begin{proof}
There is a unique regular solution $X^{s,\mu_0,\I}$ by proposition~\ref{prop:JMP}.
The rest of the proof is lemma~\ref{lem:apriori-rate2}.
\end{proof}

\section{Study of the nonlinear process with delays}\label{section:NL-DDE}

We now study the solutions of \eqref{eq:MKV} when $D>0$ by taking advantage of the results of the previous section.

\begin{proof}[Proof of theorem~\ref{th-existence-MKV-delay}]
For convenience, define the vector field $G(x) \defi (F(x_1)-x_2+I, x_1-x_2)$. Then \eqref{eq:MKV} reads for $t\in[0,D]$:
\[
X_t^{nl} = X_0 + \int_0^t G(X_u^{nl}) + J\cdot (\E\lambda(V^{nl}_{0}),0)\,\rmd u
+ \int_0^t\int_0^\infty
(\vr-V_{u^-}^{nl},\wb)\,
\indic_{\{z\leq \lambda(V^{nl}_{{u^-}})\}}
\,\mathbf{N}(\rmd u,\rmd z).
\]
We define the constant function $\I_1(t) \defi  \E\lambda(V^{nl}_{0})$ on $[0,D]$.
According to theorem~\ref{th:Knu}, there is a path-wise  unique solution $X^{(1)}$ to this equation such that  $t\to \E\lambda\left(X_t^{(1)}\right)$ is continuous.

Hence $t\to\E\lambda(V^{nl}_{t-D})$ is entirely determined on $[D,2D]$ and is continuous. We define $\I_2(t) \defi  \E\lambda(V^{(1)}_t)$ on $[D,2D]$, \eqref{eq:MKV} reads for $t\in[D,2D]$:
\[
X_t^{nl} = X_D + \int_D^t G(X_u^{nl}) + J\cdot (\I_1(u),0)\,\rmd u
+ \int_D^t\int_0^\infty
(\vr-V_{u^-}^{nl},\wb)\,
\indic_{\{z\leq \lambda(V^{nl}_{{u^-}})\}}
\,\mathbf{N}(\rmd u,\rmd z).
\]
According to theorem~\ref{th:Knu}, there is a path-wise unique solution $X^{(2)}$ to this equation such that  $t\to \E\lambda\left(X_t^{(2)}\right)$ is continuous. By recursion on the time intervals $[kD,(k+1)D]$, we conclude that there is a path-wise unique solution $X^{nl}$ to \eqref{eq:MKV} such that $t\to \E\lambda(V_t)$ is continuous.
\end{proof}

\section{Stationary distributions of the autonomous linear process}\label{section:DI}

In this section, we study the case where the current $\I$ is constant in time. Without loss of generality, we assume that $\I=0$.
The dynamics of the autonomous PDMP $X^{x,I}$ solution to \eqref{eq:EDSL} is completely determined by the dynamics of the enclosed Markov chain $(\bfw_{n},S_{n})_n$.
We write $\bfG$ its transition operator, for $h\in \mathrm L ^\infty(\R^2\times\R_+)$:

 \[
	\bfG_I h(x_0,s_0)\defi\int_{\mathbb R_+} p^I(t,x_0)h(\Delta\Phi^t(x_0,I),t)dt.
\]
Moreover $(\bfw_{n})_n$ is also a Markov chain with transition operator:
\begin{equation}
	\label{eq.U}
	\bfU_I h(x_0)
	\defi
	\int_{\R_+} p^I(t,x_0)h(\Delta\Phi^t(x_0,I))dt.
\end{equation}
When the context allows it, we remove the $I$ dependency in the notations. We shall also write $\bfG_I h(w_0,s_0)$ and $\bfU_I h(w_0,s_0)$ with $h\in\mathrm L ^\infty(\Ew\times \R_+)$ when the enclosed chain lies on the reset axis $\CR$, \textit{e.g.} when $X^{x,I}_0$ belongs to $\CR$.

We study the invariant distributions of \eqref{eq:EDSL} by looking at the invariant distributions $\mu^{\mathbf w,S}$ of the enclosed Markov chain. We then derive the invariant distribution of $X^{x,I}$ based on $\mu^{\mathbf w,S}$ recovering a result by Costa \cite{costa_stationary_1990} under different assumptions. We also show the absolute continuity of the invariant distribution, see \cite{benaim_qualitative_2015, locherbach_absolute_2018} for related studies and especially \cite{locherbach_absolute_2018}.

In the next section, the results are derived with locally uniform bounds in $I$ which makes them a bit technical, this is necessary for section~\ref{section:NLDI} where we need to show that $\mu^{\mathbf w,S}$ is continuous in $I$ for some weighted total variation norm on the space of Radon measures.

\subsection{Existence of invariant distribution for the enclosed chain}
In the next proposition, we derive a Doeblin estimate for the enclosed chain. This estimate is relatively easy to obtain when $\bfw_n$ is ``large'' because the flow is simple in this case: $w$ decreases between jumps. We then appeal to this case when $\bfw_n$ is ``small'' by using many successive jumps which are of constant size $\wb$ and which makes $\bfw_n$  ``large''. This justifies the need to compute $\bfG^k$.
To this end, it is convenient to define:
\[
	\psi_t(x_0) \defi  \Delta\Phi^t(x_0) = (\vr,w_{x_0}(t)+\wb)
\]
and
\[
	\psi^k_{t_{1:k}}(x_0) \defi
	\psi_{t_k}\circ\psi_{t_{k-1}}\circ\cdots\circ\psi_{t_1} (x_0)
\]
where we use the notation $t_{1:k}=(t_1,\dots,t_k)$. Hence, conditionally on $S_{1:k}=t_{1:k}$
and $X_0=x$, we have $(\vr,\bfw_{k})=\psi^k_{t_{1:k}}(x)$.
Note that for $1<\ell<k$:
\[
	\psi^k_{t_{1:k}}(x) = \psi^{k-\ell}_{t_{\ell+1:k}}(\psi^\ell_{t_{1:\ell}}(x)).
\]
Let $p_{k}(t_{1:k},x)$ be the conditional density of $S_{1:k}$ given $X_0=x$:
\begin{align*}
	p_{k}(t_{1:k},x)
	&\defi
	\P(S_{1:k} = t_{1:k}|X_0=x).
\end{align*}
Note that for $k>1$
\begin{align*}
	p_k(t_{1:k},x)
	=
	p\bigl(t_k,\psi^{k-1}_{t_{1:k-1}}(x)\bigr)\,p_{k-1}(t_{1:k-1},x)
\end{align*}
which allows to write concisely
\begin{equation*}
	\bfG^kh(x,s) = \int_{(\R_+)^k}p_k(t_{1:k},x)h\left(\psi^k_{t_{1:k}}(x),t_k\right)dt_{1:k}.
\end{equation*}

\begin{proposition}
	\label{prop-doebolin2}
	Grant assumptions~\ref{hyp.F},~\ref{as:ratepos-increasing} and assume that $\lambda$ is continuous. Let us consider a finite interval $[\underline I, \bar I]$ and a family of partitions $\CPf{\underline{I}}{\bar I}$ from lemma~\ref{lemm_partition} with $\CP^I=(x_{sep}^I,w_{23})$. Finally, assume that 
	\begin{equation}
		\int^\infty_{w_{23}}\frac{\lambda(\Vn^-(w,\bar I))}{w-\Vn^-(w,\bar I)}dw<\infty.
	\end{equation}
	There is $w_{23}$ high enough such that the following holds.
	There exist a probability distribution $\nu_D^I$ (only depending on $I$) on $\R\times\R_+$,  $k_D\in\mathbb N^*$ and $\beta_D\in(0,1)$ such that for all $h\in \mathrm L ^\infty(\Ew\times \R_+;\R_{+})$, $s_0\geq 0,\ I\in [\underline I, \bar I]$
	\begin{align}
		\label{eq.prop-doebolin2.1}
		\forall w_{0}\in[w^*, 2w_{23}],\  \bfG_I^{k_D} h(w_{0},s_0) \geq \beta_D \, \nu_D^I(h)
	\end{align}
	and
	\begin{align}
		\label{eq.prop-doebolin2.2}
		\forall w_{0}\geq 2\, w_{23},\ \bfG_I^2 h(w_{0},s_0) \geq \beta_D \,\nu_D^I(h).
	\end{align}
Furthermore, $(w_{23}+2\wb,2w_{23}+2\wb)$ is in the support of the first marginal $\tilde \nu^I_D$ of $\nu^I_D$ such that $\tilde\nu^I_D(A)\defi\nu^I_D(A\times\R_+)$.
\end{proposition}

\begin{proof}
	We recall that $w_{23}, v_{34}$ and $w^*$ are the same for all $\CP^I = (x^I_{sep}, w_{23})$.
	We first establish \eqref{eq.prop-doebolin2.2}.
	For all $w_{0}\geq 2\, w_{23}$, we consider the solution $(v^I_{w_{0}}(t),$ $w^I_{w_0}(t))$ from $(\vr,w_{0})$ and let $\tau^I_1(w_{0})$ (resp. $\tau^I_2(w_{0})$) be the first time $w_{w_{0}}$ reaches $2\,w_{23}$ (resp. $w_{23}$). We have:
	\begin{align*}
		\bfG_I^2 h(w_0,s_0) &= \int_0^\infty\int_0^\infty dt_1dt_2\ p^I(t_1,w_0)p^I(t_2,\Delta\Phi^{t_1}(w_0,I))h(\psi^2_{t_{1:2}}(w_0,I),t_2)\\
		&\geq
		\int_0^\infty\int_0^\infty dt_1dt_2\ p^I(t_1,w_0)p^I(t_2,\Delta\Phi^{t_1}(w_0,I))h(\psi^2_{t_{1:2}}(w_0,I),t_2)\indic_{[w_{23},2w_{23}]}(w^I_{w_0}(t_1)).
	\end{align*}
	We denote by $\vmin^I$ the $v$-component of the separatrix $x^I_{sep}$ when it crosses $\{w=2w_{23}\}$, $\vmin^I$ exists thanks to proposition~\ref{prop:separatrix}. We then define $\vmin = \min\limits_{I\in[\underline I, \bar I]}\vmin^I$.

	\begin{lemma}
		There is a constant $C_\Lambda>0$ such that for all $w_{0}\geq 2\, w_{23}, I\in  [\underline I, \bar I]$ and for all $t\in [\tau^I_1(w_0),\tau^I_2(w_0)]$,
		\[p^I(t,w_0)\geq K\defi \lambda(\vmin)\,\exp\left(-C_\Lambda\right)>0.\]
	\end{lemma}
	\begin{proof}
		Since $\lambda$ is non decreasing by assumption~\ref{as:ratepos-increasing}, $ \forall t\in [\tau^I_1(w_0),\tau^I_2(w_0)]$, $\lambda(v^I_{w_0}(t))\geq \lambda(\vmin)$ because the solution stays above the separatrix. Hence for $t\in \left[\tau^I_1,\tau^I_2\right]$:
		\[
			p^I(t,w_0) \geq  \lambda(\vmin)\exp\left(-\int_0^{\tau^I_{2}(w_0)} \lambda(v^I_{w_0}(u))du\right).
		\]
		We first prove that 
		\begin{equation}\label{eq:tau2bounded}
			\sup\limits_{I\in [\underline I, \bar I],\ w_0\geq 2w_{23}}\int_0^{\tau^I_{2}(w_0)} \lambda(v^I_{w_0})<\infty.
		\end{equation}
		Let $\tau^I_v(w_0)<\infty$ be the hitting time of the $v$-nullcline from $(\vr, w_0)$ with $w_0\geq w_{23}$. It is finite by lemma~\ref{lemma:whinf} if $w_{23}$ is high enough which we now assume.
		\begin{itemize}
			\item \textbf{		Case $\tau_v^I> \tau_2^I$.}
			Using the backward flow from $(v_{23}, w_{23})$, one can show that there is $\tilde w^I$, continuous function of $I$, such that $\tau_v^I> \tau_2^I$ if and only if $w_{23}\leq w_0\leq \tilde w^I$. $\tau_2^I(w_0)$ is continuous in $I
			$ and $w_0$, hence 
			\[
			\sup \limits_{I\in[\underline I,\bar I],\ w_0\in[2w_{23}, \tilde w^I],\ 2w_{23}\leq \tilde w^I}\tau_2^I<\infty.
			\] 
			This implies \eqref{eq:tau2bounded} because the integrand is bounded by $\lambda(\vr)$.
			\item \textbf{		Case $\tau_v^I\leq \tau_2^I$.}
		We have
		\[
		\int_{0}^{\tau^I_{2}(w_0)} \lambda(v^I_{w_0}) = \int_{0}^{\tau^I_v(w_0)} \lambda(v^I_{w_0}) + \int_{\tau^I_v(w_0)}^{\tau^I_{2}(w_0)} \lambda(v^I_{w_0}).
		\]
		From proposition~\ref{prop.tnbounded}, we find 
		\[
			\int_0^{\tau^I_v(w_0)} \lambda(v^I_{w_0})\leq \lambda(\vr)\sup\limits_{I\in [\underline I, \bar I],\ w_0\geq 2w_{23}}\tau^I_v(w_0)<\infty.
		\]
		The solution being away from the $w$-nullcline, we can write:
		\[
		\big(v^I_{w_0}(t),w^I_{w_0}(t)\bigr)
		=
		\big(V^I_{w_0}(w^I_{w_0}(t)),w^I_{w_0}(t)\bigr)
		\]
		where $V^I_{w_0}$ is defined in \eqref{eq:Vx}.
		In $\int_{\tau^I_v(w_0)}^{\tau^I_{2}(w_0)} \lambda(v^I_{w_0})$, the change of variable $u=w^I_{w_0}(t)$:
		\[
		\rmd u
		=
		\big(v^I_{w_0}(t)-w^I_{w_0}(t)\bigr)\, \rmd t
		=
		\big(V^I_{w_0}(w^I_{w_0}(t))-w^I_{w_0}(t)\bigr)\, \rmd t
		\]
		gives:
		\[
		\int_{\tau^I_v(w_0)}^{\tau^I_2(w_0)} \lambda(v^I_{w_0})
		=
		\int_{2w_{23}}^{w^I_{w_0}(\tau^I_v(w_0))} \frac{\lambda(V^I_{w_0}(u))}{u-V^I_{w_0}(u)}\,\rmd u.
		\]
		Also, for $t\in \left[\tau^I_v(w_0),\tau^I_2(w_0)\right]$, the solution lies below the $v$-nullcline, hence
		\[
		\int_{\tau^I_v(w_0)}^{\tau^I_2(w_0)} \lambda(v_{w_0})\leq  \int_{2w_{23}}^{\infty} \frac{\lambda(\Vn^-(u,I))}{u-\Vn^-(u,I)}\,\rmd u\leq  \int_{w_{23}}^{\infty} \frac{\lambda(\Vn^-(u,\bar I))}{u-\Vn^-(u,\bar I)}\,\rmd u<\infty
		\]
		where the before last inequality stems from $\lambda$ non decreasing and $I\to \Vn^-(u,I)$ increasing.
		This proves \eqref{eq:tau2bounded}.
		\end{itemize}				
		We thus have proved that \[\forall w_{0}\geq 2\, w_{23}, I\in  [\underline I, \bar I], t\in [\tau^I_2(w_0),\tau^I_1(w_0)],\quad p^I(t,w_0)\geq \lambda(\vmin)\,
		\exp\left(
		-C_\Lambda
		\right)>0\]
		for some $C_\Lambda>0$.
	\end{proof}
	We can now lower bound $\bfG_I^2h$ using the change of variable $u=w^I_{w_0}(t_1)$:
	\begin{align*}
		\bfG_I^2 h(w_0,s_0) &\geq K \int_{\tau^I_1(w_0)}^{\tau^I_2(w_0)} dt_1\int_0^\infty dt_2\ p^I(t_2,\Delta\Phi^{t_1}(w_0,I))h(\psi^2_{t_{1:2}}(w_0,I),t_2)\\
		&=K
		\int_0^\infty dt_2\int_{w_{23}}^{2w_{23}} h(\Delta\Phi^{t_2}(\Delta(\vr,u),I),t_2)\frac{p^I(t_2,\Delta(\vr,u))}{u-V^I_{w_0}(u)}du\\
		&\geq K
		\int_0^\infty dt_2\int_{w_{23}}^{2w_{23}} h(\Delta\Phi^{t_2}(\Delta(\vr,u),I),t_2)\frac{p^I(t_2,\Delta(\vr,u))}{d}du
	\end{align*}
	where $d$ is the distance of $(v_{min},2w_{23})$ to the $w$-nullcline. 
	If we introduce the probability distribution 
	
	\begin{equation}\label{eq:nuD}
	\nu_D^I(h) \defi \int_{w_{23}}^{2w_{23}}\frac{du}{w_{23}}\int_0^\infty p^I(t_2,(\vr,u+\wb))h(\Delta\Phi^{t_2}((\vr,u+\wb),I),t_2)dt_2,
	\end{equation}
	we then have proved \eqref{eq.prop-doebolin2.2}. From the flow property, one obtains that $(w_{23}+2\wb,2w_{23}+2\wb)\subset\text{supp}(\tilde\nu_D^I)$.

	We now prove \eqref{eq.prop-doebolin2.1}. The space $\Ew\times \R_+$ is invariant by the Markov Chain $(\mathbf w_n^I,S_n^I)_n$. Let $k$ be the smallest integer such that $w^*+ k\,\wb>2\, w_{23}$.
	We look at $\bfG_I^{k+2}h(w_0,s_0)$ and try first to lower bound the terms involving $p^I(t,w)$. For $\epsilon>0$ small enough such that the $t_i$ are smaller than the deterministic explosion times:
	\[
		\forall w_0\in\left[w^*,2w_{23}\right],\forall s_0\geq 0,\quad \bfG_I^{k +2}h(w_0,s_0) \geq
		\int_{[0,\epsilon]^{k}}\rmd t_{1:k} \; \; p_k^I(t_{1:k},w_0)\,
		\bfG_I^2 h\bigl(\psi^{k}_{t_{1:k}}(w_0,I),t_k).
	\]
	
	\begin{lemma}\label{lemma:hadamard}
		There is $\epsilon>0$ such that for all $(t_i)_{i=1,\cdots,k}\in[0,\epsilon]^k$, $w_0\in[w^*,2w_{23}]$ and $I\in[\underline{I},\bar I]$, we have \[\langle\psi^{k}_{t_{1:k}}(w_0,I),\begin{bmatrix} 0 \\ 1 \end{bmatrix}\rangle>2w_{23}.\]
	\end{lemma}
\begin{proof}
	The existence of $\epsilon$ such that $\psi^{k}_{t_{1:k}}(w_0,I)$ is well defined is a consequence of the Cauchy-Lipschitz theorem around $(\vr,w_0+l\wb)$ for $w_0\in[w^*,2w_{23}]$ and $l=1,\cdots,k$.
	
	Since $\psi^{k}_{t_{1:k}}$ is $C^1$ with respect to the $t_i$s, $w_0$ and $I$ using Hadamard's lemma , we can write for $((t_i)_i,w_0,I)\in [0,\epsilon]^k\times[w^*,2w_{23}]\times[\underline{I},\bar I]$
	\[
		\langle\psi^{k}_{t_{1:k}}(w_0,I),\begin{bmatrix} 0 \\ 1 \end{bmatrix}\rangle = \sum\limits_{i=1}^kt_ih_i(t_{1:k},w_0,I)+w_0+k\wb.
	\]
	for some continuous functions $h_i$.
	It follows that
	\[
		\langle\psi^{k}_{t_{1:k}}(w_0,I),\begin{bmatrix} 0 \\ 1 \end{bmatrix}\rangle\geq w^*+k\wb - \epsilon\sum\limits_{i=1}^k\norm{h_i}_\infty >2w_{23}
	\]
	where the last inequality is for $\epsilon$ small enough.
\end{proof}
From the continuity of $\lambda$, $(t_{1:k},w_0,I)\to p_k^I(t_{1:k},w_0)$ reaches its minimum $K_1>0$ on $[0,\epsilon]^{k}\times [w^*,2w_{23}]\times[\underline{I},\bar I]$.
Hence, we find that:
	\[
	\forall I\in[\underline{I},\bar I], \quad\bfG_I^{k +2}h(w_0,s_0) \geq K_1 \int_{[0,\epsilon]^{k}} \rmd t_{1:k}\ \bfG_I^2 h\bigl(\psi^{k}_{t_{1:k}}(w_0,I),t_k).
	\]
	Using the lemma~\ref{lemma:hadamard}, we can appeal to the previous case and find:
	\[
		\forall w_0\in\left[w^*,2w_{23}\right],\ \bfG_I^{ k +2}h(w_0,s_0) \geq K_1\epsilon^k\beta_D\nu_D^I(h)
	\]
	which concludes the proof.
\end{proof}

We show that the Markov chain $(\bfw_n)_n$ is positive Harris recurrent.

\begin{theorem}
	\label{thm:existence-uniq-inv2}
	Grant the assumptions of proposition~\ref{prop-doebolin2}.
	Then $\bfG^{2k_D}_I$ satisfies a global Doeblin condition on $\Ew\times\R_+$.
	We write $\mu^{\mathbf w,S}_I$ its unique invariant distribution. 
	There are $\alpha \in(0,1)$ and $C>0$ independent of $I$ such that if $(w_0,S_0)$ has law $\nu_0$ with $supp(\nu_0)\subset \Ew\times\R_+$,
	\[
		\forall I\in[\underline I, \bar I],\quad\norm{\bfG^{*n}_I\nu_0-\mu^{\mathbf w,S}_I}_{TV}\leq C\alpha^n\norm{\nu_0-\mu^{\mathbf w,S}_I}_{TV}.
	\]
\end{theorem}

\begin{proof}
	Let us show that $\bfG^{2k_{D}}_I$ satisfies a global Doeblin condition. We write $G_I$ the transition measure of the Markov chain $(\bfw_n,S_n)_n$. We have by proposition~\ref{prop-doebolin2} for $B\in\mathcal B(\Ew\times \R_+)$, $$ \forall w_{0}\geq 2\, w_{23},s_0\geq 0,\ G^2((w_0,s_0),B)\geq \beta_D \nu_D(B).$$ We define $\tilde \CP_2 \defi [2w_{23},\infty)\times\mathbb R_+$. We then find
	for all $w_{0}\geq 2\, w_{23},\ s_0\geq 0$
	\begin{equation*}
		G_I^4((w_0,s_0),B) \geq \int_{\tilde \CP_2} G_I^2((w_0,s_0),dw,ds)G_I^2((w,s),B)\geq \beta_D \nu^I_D(B) G_I^2((w_0,s_0),\tilde \CP_2)\geq \beta_D^2\nu^I_D(\tilde\CP_2)\nu^I_D(B).
	\end{equation*}
	It follows that
	\begin{equation*}
		G_I^{2k_{D}}((w_0,s_0),B) \geq \beta^{1+k_{D}}_D\nu^I_D(\tilde\CP_2)^{k_{D}}\nu^I_D(B).
	\end{equation*}
	We can lower bound 
	\[I\to\nu^I_D(\tilde\CP_2) = \int_{w_{23}}^{2w_{23}}\frac{du}{w_{23}}\mathbf U_Ih((\vr,u+\wb)),\quad h(x_2) = \mathbf 1_{[2w_{23},\infty)}(x_2)\]
	 by continuity using Lebesgue's dominated convergence theorem and lemma~\ref{lemma:Utau}.  This lower bound is positive by proposition~\ref{prop-doebolin2}.
	Hence, combining with the case $w_{0}\leq 2\, w_{23},\ s_0\geq0$, we find that there is a constant $\beta_1>0$ independent of $I$ such that
	\[\forall w_0\geq w^*,\ s_0\geq0,\  G_I^{2k_{D}}((w_0,s_0),B) \geq \beta_1\nu_D^I(B). \]
	Thus, $G_I^{2k_{D}}$ satisfies a global Doeblin condition from which the theorem follows.
\end{proof}

\noindent It implies that $(\bfw_n)_n$ is geometric ergodic as well, we write $\muW_I$ its unique invariant distribution.
In order to estimate the tail properties of $\muW_I$, we rely on a, locally uniform in $I$, Lyapunov function.


\begin{lemma}\label{lem:er}
	Grant the assumptions of proposition~\ref{prop-doebolin2}. Then there are $r>0, K_L\geq0$ and $\gamma_L\in(0,1)$ such that $V\defi e^{r\cdot}$ satisfies
	\begin{equation}\label{eq:lyap}
		\forall w_0\in \Ew, \ \forall I\in[\underline I, \bar I],\quad \mU_I V(w_0)\leq \gamma_L V(w_0) + K_L.
	\end{equation}
	As a consequence $\muW_I(e^{r\cdot})<\infty$ whence $\muW_I(w,\infty)=O(e^{-rw})$.
\end{lemma}
\begin{proof}
	We first prove \eqref{eq:lyap}.
	We call $\tau_3(w_0,I)$ the time it takes for a solution from $(\vr,w_0)\in\CP_2^I$ to reach the set $\CP_3^I$.
	For $w_0 \geq w_{23}$ and using the monotony of $V$:
	\begin{multline*}
		\mU_I V(w_0)
		=
		e^{r\wb}\int_0^{\tau_3(w_0,I)}V(w_{w_0}^I(t))p^I(dt,w_0)+e^{r\wb}\int_{\tau_3(w_0,I)}^{\te{w_0,I}} V(w^I_{w_0}(t))p^I(dt,w_0)
		\\
		\leq e^{r\wb}V(w_0) \P\left(T_1^{w_0,I}<\tau_3(w_0,I)\right)+e^{r\wb}V(w_{max})
		\leq \gamma_L V(w_0)+K_L
	\end{multline*}
	where
	\[w_{max}\defi\sup\limits_{I\in[\underline I, \bar I],\ x\in \CP_3^I\cap\CP_4^I,\ t<\te{x,I}}w^I_x(t)<\infty,\quad K_L\defi e^{r\wb}V(w_{max})\] 
	and
	\[\gamma_L \defi e^{r\wb}\sup\limits_{I\in[\underline I, \bar I],\ w_0>w_{23}}\P\left(T_1^{w_0,I}<\tau_3(w_0,I)\right).\]
	Note that $w_{max}$ is finite thanks to theorem~\ref{th:jtb1}, see proof of corollary~\ref{coro:wbound}.
	We want to show that we can find $r>0$ such that $\gamma_L\in(0,1)$. This is possible provided that the supremum in $\gamma_L$ is strictly smaller than $1$ or that
	\[
	\sup_{I\in[\underline I, \bar I],\ w_0>w_{23}} \int_0^{\tau_3(w_0,I)}\lambda(v^I_{w_0}(s))\,\rmd s < \infty.
	\]
    This was shown in the proof of proposition~\ref{prop-doebolin2}.  From the bound
	\[
	\forall w_0\in[w^*, w_{23}],\ \forall I\in[\underline I, \bar I],\ \mU_IV\leq e^{r\wb}V(w_{max}),
	\]
	we conclude that \eqref{eq:lyap} is true.
	We then have
	\[
	\forall w_0\in \Ew,\ \forall I\in[\underline I, \bar I],\ \mU_I V(w_0)\leq V(w_0) -(1-\gamma_L)V(w_0) + K_L.
	\]
	The conclusion $\muW_I(V)<\infty$ follows from \cite{meyn_stability_1992}[theorem 14.3.7].
	Finally:
	\[
	\int_w^\infty \muW_I = 	\int_w^\infty \frac{V(x)}{V(x)}\muW_I(dx) \leq e^{-r w}\muW_I(V)
	\]
	which concludes the proof of the lemma.
	
\end{proof}

\begin{lemma}
	\label{lemma:ET1}
	Grant assumptions~\ref{hyp.F}, \ref{as:ratepos-increasing}. Let us consider a finite interval $[\underline I, \bar I]$ and a family of partitions $\CPf{\underline{I}}{\bar I}$ from lemma~\ref{lemm_partition}.  Then there is a constant $M>0$ such that
	\begin{equation}\label{eq:T1x}
		\forall I \in [\underline I, \bar I], \quad\forall x\in\CP_2^I\cup\CP_3^I,\quad \E\left(T^{x,I}_1\right) \leq M + \log\left(\frac{x_2-\vr}{w_{23}-\vr}\right)\indic_{[w_{23},\infty)}(x_2).
	\end{equation}
	Further assume that $\int^\infty_{w_{23}}\frac{\lambda(\Vn^-(w,\bar I))}{w-\Vn^-(w, \bar I)}dw<\infty$ and that $\lambda$ is continuous, then
	\[\E_{\muW_I}(T_1^{x,I})<\infty.\]
\end{lemma}
\begin{proof}
	Let $I\in [\underline I, \bar I]$.
	Note that $\E\left(T^{x,I}_1\right) = \int_0^{\te{x,I}}e^{-\Lambda^I(x,t)}\,\rmd t$. We first look at the case $x\in\CP_3^I$.
	Either the trajectory from $x$ remains in $\CP_3^I$ or it explodes through $\CP_4^I$. In all cases
	\[\E\left(T^{x,I}_1\right) \leq \int_0^{\te{x,I}} e^{-\lambda(v_{23}^I)t} dt \leq \frac{1}{\lambda(v_{23}^I)}\leq \frac{1}{\lambda(v_{min})}\]
	where $\vmin\defi \min\limits_{I\in [\underline I, \bar I]}\vmin^I$.
	When $x$ belong to $\CP_2^I$, we find that $\E\left(T^{x,I}_1\right)  \leq \tau_3(x,I)+\E T_1^{{\Phi^{\tau_3}(x)}}\,\indic_{T_1^{x}>\tau_3(x,I)}$ where $\tau_3(x,I)$ is the time it takes to reach $\CP_3^I$ from $x$.  We note that $\E T_1^{{\Phi^{\tau_3}(x)}}\,\indic_{T_1^{x}>\tau_3(x,I)}$ is bounded by $\frac{1}{\lambda(v_{23})}$. Also, in $\CP_2^I$, $\dot w \leq \vr-w$ from which it follows that $\tau_3(x,I)\leq \log\left(\frac{x_2-\vr}{w_{23}-\vr}\right)$. We thus have proved that
	\[ \forall x\in\CP_2^I\cup\CP_3^I, \quad\forall I \in [\underline I, \bar I],\quad \E\left(T^{x,I}_1\right) \leq \frac{1}{\lambda(v_{min})} + \log\left(\frac{x_2-\vr}{w_{23}-\vr}\right)\indic_{[w_{23},\infty)}(w_0). \]
	It follows from lemma~\ref{lem:er} that $\E_{\muW_I}(T_1^{x,I})<\infty$ because $e^{-r\cdot}\log(\cdot-\vr)$ is bounded on $[w_{23},\infty)$.
\end{proof}

\subsection{Properties of the invariant distribution of the enclosed chain}

If $(\vr, \vr)$ is an equilibrium of \eqref{eq:micro-unique}, it gives the transition kernel $\mU_I((\vr, \vr), \cdot) = \delta_{(\vr,\vr+\wb)}$. In order to prove that the transition measure has a density, 
we remove the case where an equilibrium belongs to the reset line $\CR$ from our analysis.

\begin{proposition}\label{prop:K}
	Grant assumptions~\ref{hyp.F},\ref{as:ratepos-increasing}.
	For all $I\in\Ireg$, let $\CP^I$ be a partition. For $w_0\in\mathbf E_w^I$, there is a measurable application $K_I:\Ew^I\times\Ew^I\to\R_+$ such that
	for all $h\in\mathrm L^{\infty}(\Ew^I, \mathbb R^+)$:
	\[
	\mU_I h(w_0) = \int_{\Ew }K_I(w_0,w)h(w+\wb)dw.
	\]
	Moreover, $K_I$ satisfies the  following properties:
	\begin{enumerate}
		\item $K_I(w_0,\cdot)\in \mathrm L^{1}(\Ew^I,\R_+ )$
		\item for all $w_0\in \Ew^I$, $\int_{\Ew^I}K_I(w_0,\cdot)=1$
		\item for all $w_0$ large enough, $\forall\epsilon > 0$, $K_I(w_0,w_0+\epsilon) = 0$.
	\end{enumerate}
\end{proposition}
\begin{proof}
	We consider $I\in\Ireg$, $h\in\mathrm L^{\infty}(\Ew^I, \R_+ )$ and $w_0\in \Ew^I$. We introduce the increasing sequence of deterministic hitting times  $(\tau_n(w_0,I))_{n\geq  1}$ of the $w$-nullcline from $(\vr,w_0)$ with $\tau_0(w_0,I)\defi 0$ and denote by $N_h(w_0,I)$ the number of hitting times. If $N_h(w_0,I)$ is finite, we define $\tau_{N_h(w_0,I)+1}(w_0,I)\defi \te{w_0,I}$.
	Finally, we write $w^{(n)}(w_0,I)$ the value of the $w$-component at the hitting time $\tau_n$ and define $w^{(0)}(w_0,I)\defi w_0$. When $N_h$ is finite, we define $w^{(N_h+1)}(w_0,I)\defi \lim\limits_{t\to t_\infty}w_{w_0}(t,I)$. We find
	\[
	\mU_I h(w_0) = \sum\limits_{n=0}^{N_h(w_0,I)}\int_{\tau_n(w_0,I)}^{\tau_{n+1}(w_0,I)} p^I(t,w_0)h(w^I_{w_0}(t)+\wb)dt.
	\]
	We want to perform the change of variables $w=w^I_{w_0}(t)$ which is legitimate when $(\vr, w_0)$ is not an equilibrium point - \textit{e.g.} $I\in\Ireg$ - and when $w\neq w^{(n)}(w_0,I)$ for all $n$. We then write the trajectory from $(\vr,w_0)$ for $t\in(\tau_n,\tau_{n+1})$ as 
	\[
		(V^{(n)}(w^I_{w_0}(t)), w^I_{w_0}(t))
	\]
	where $V^{(n)}$ is the solution of \eqref{eq:Vx} with initial condition $V^{(n)}\left(\frac{w^{(n)}+w^{(n+1)}}{2},I\right)=v^{(n/2)}$ where $v^{(n/2)}$ is such that $\left(v^{(n/2)},\frac{w^{(n)}+w^{(n+1)}}{2}\right)$ is the unique point on the solution from $(\vr,w_0)$ for $t\in(\tau_n,\tau_{n+1})$. We use this particular initial condition so that $V^{(n)}$ is defined on the interval $(w^{(n)}, w^{(n+1)})$.
	
	We thus find that the previous integral equals to 
	\[
		\int_{w^{(n)}}^{w^{(n+1)}}K_{I,n}(w_0,w)h(w+\wb)dw
	\]
	where
	\[
	K_{I,n}(w_0,w)  \defi \frac{e^{-\Lambda(w_0,\tau_n)}	 }{|w-V^{(n)}(w,I)|}\lambda\left(V^{(n)}(w,I)\right)\exp\left(-\int_{w_n}^{w}\frac{\lambda(V^{(n)}(u,I))}{|u-V^{(n)}(u,I)|}du\right)\mathbf 1_{A_n}(w)
	\]
	with $A_n\defi \left(w^{(n)},w^{(n+1)}\right)$ if $w^{(n)}\leq w^{(n+1)}$ and $\left(w^{(n+1)},w^{(n)}\right)$ otherwise. 
	Let us then define the measurable application:
	\[
	K_I(w_0,w) \defi \sum\limits_{n=0}^{N_h(w_0,I)}K_{I,n}(w_0,w).
	\]
	We now show that $K_I(w_0,\cdot)\in \mathrm L^{1}(\Ew^I )$. To this end, we define the non-decreasing sequence $(K^{\ell}(w_0,\cdot))_{\ell}$ in $\mathrm L^{1}(\Ew^I)$ where $K^{\ell}(w_0,w)\defi \sum\limits_{n=0}^{N_h(w_0,I)\wedge\ell}K_{I,n}(w_0,w)$. We have $\sup\limits_\ell\int K^\ell(w_0,\cdot)\leq 1$. From the Beppo Levi's monotone convergence theorem, $K^{\ell}(w_0,\cdot)$ converges to $K_{I}(w_0,\cdot)<\infty$ \textit{a.e.}, $K_I(w_0,\cdot)\in \mathrm L^{1}(\Ew^I)$ and $\norm{K^\ell(w_0,\cdot)-K_{I}(w_0,\cdot)}_1\to0$.
	The proof of the first two properties (i)-(ii) is straightforward because $p^I(dt,(\vr,w_0))$ is a probability measure. For property (iii), we observe that $w_{w_0}(t)$ decreases when $w_0\geq w_{23}$ and no trajectory from $\CP_4$ contributes to $K_I$.
\end{proof}

\begin{lemma}\label{lemma:pw}
	Grant the assumptions of proposition~\ref{prop-doebolin2} and proposition~\ref{prop:K}.
	Then, the invariant distribution $\muW_I$ of $(\bfw_n)_n$ has density $\pW_I\in\mathrm L^1(\Ew^I, \mathbb R^+ )$ with respect to the Lebesgue measure for all $I\in\Ireg$. It satisfies
	\begin{equation}\label{eq:integralequation}
		\forall u\in\Ew^I, \quad \pW_I(u+\wb) = \int_{\Ew^I}K_I(w_0,u)\pW_I(w_0)dw_0.
	\end{equation}
\end{lemma}
\begin{proof}
By  theorem~\ref{thm:existence-uniq-inv2}, $\muW$ exists and is $\sigma$-finite. 
Consider $h\in \mathrm L^{\infty}(\mathbf E_w^I, \R_+)$, we apply the Fubini-Tonelli theorem and obtain
\[
\muW_I(h) = \int\left[\int K_I(w_0,w)h(w+\wb)dw\right] \muW_I(dw_0)= \int\left[\int K_I(w_0,w)\muW_I(dw_0)\right]h(w+\wb)dw
\]
where the last integral in brackets is a measurable mapping.
This proves that $\muW_I$ has a density. The rest of the lemma is straightforward.
\end{proof}

We now derive an integral equation to estimate the density $\pW_I$.
\begin{corollary}
	Let us define $w_4^{I,\infty} \defi \lim\limits_{t\to t_\infty}w_{(v_{34},w_{23})}(t;I)$ which is finite by theorem~\ref{th:jtb1}.
	The integral equation \eqref{eq:integralequation} reads
		\begin{equation}\label{eq:pwr}
	\forall w> \max(w_{23},w_4^{I,\infty}),\quad \pW_I(w+\wb) = \int_{w}^\infty \pW_I(w_0)\,r(w_0,w)  \rmd w_0
	\end{equation}
with
\[
\forall w\geq w_0,\quad r(w_0,w)\defi\frac{\lambda(V_{w_0}(w))}{w-V_{w_0}(w)}\exp\left(-\int_w^{w_0}\frac{\lambda(V_{w_0}(u))}{u-V_{w_0}(u)}du\right).
\]
\end{corollary}
\begin{proof}
		We consider the invariant distribution $\muW_I$ and $w> \max(w_{23},w_4^{I,\infty})$. Given the vector field, there are two ways to reach $w$ from $(\vr,w_0)\in\CR_\CP$.
		Either it is reached from $\CP_2$ or it is reached from below in $\CP_4$. The second possibility is incompatible with $w> \max(w_{23},w_4^{I,\infty})$. Hence:
		\[
				\pW_I(w+\wb) = \int_{w}^\infty K_I(w_0,w)\pW_I(w_0)dw_0.		
		\]
		It is now straightforward to adapt the proof of proposition~\ref{prop:K} to get \eqref{eq:pwr}.
\end{proof}

\begin{proposition}	\label{prop:mu-support}
	Grant the assumptions of proposition~\ref{prop-doebolin2} and proposition~\ref{prop:K}.
	The invariant distribution $\muW_I$ has density $\pW_I$ with respect to the Lebesgue measure for $I\in\Ireg$. Further there is $M>0$ depending on $I$ such that $\pW_I$ is continuous on $[M,\infty)$ with  asymptotic behavior
	\[ \pW_I(u) =o\left(\frac1u\right).\]
	Finally, $\pW_I$ is not compactly supported and $\pW_I$ is positive on $[M,\infty)$.
\end{proposition}

\begin{proof}
	From theorem~\ref{thm:existence-uniq-inv2}, $(\bfw_n)_n$ has a unique invariant distribution $\muW$. From lemma~\ref{lemma:pw}, this distribution has density $\pW\in\mathrm L^1_+(\Ew )$ with respect to the Lebesgue measure.
	Owing to the flow properties in $\CP_2$, one gets
	
	\[
	\forall w_0\geq w,\ r(w_0,w)\leq\frac{\lambda(\vr)}{w-\vr}\leq \frac{\lambda(\vr)}{w_{23}-\vr}.
	\]
	Let us show that $\pW$ is continuous on $(w_{23}+w_4^{I,\infty},\infty)$. We fix $w>  w_{23}+w_4^{I,\infty}$ and consider $u>0$ small enough. Then using \eqref{eq:pwr} :
	\begin{equation*}
		|\pW(w+u+\wb)-\pW(w+\wb)| \leq \int_{w+u}^\infty \pW( w_0 )|r( w_0 ,w+u)-r( w_0 ,w)| \rmd w_0+ \int_{w}^{w+u} \pW( w_0 )r( w_0 ,w)\rmd w_0.
	\end{equation*}
	The last term tends to zero as $u$ tends to zero because $\pW$ is integrable and $r$ is bounded. 
	The first term tends to zero by Lebesgue's dominated convergence theorem. 
	The case $u<0$ is similar.
	This shows that $\pW$ is continuous.
	 The asymptotic behavior comes from 
	 \[\pW(u+\wb)\leq \frac{\lambda(\vr)}{u-\vr}\int_u^\infty \pW.\]
From proposition~\ref{prop-doebolin2}: $(w_{23}+2\wb,2w_{23}+2\wb)\subset\text{supp}(\tilde\nu_D)\subset\text{supp}(\muW)$ and $\muW$ does not have compact support because $w_{23}$ can be chosen arbitrarily large. It also implies that $\pW(w)$ is positive for $w\geq M$.
\end{proof}

\begin{proof}[Proof of theorem~\ref{th-mai-results-muW}]
	The proof is a consequence of theorem~\ref{thm:existence-uniq-inv2}, proposition~\ref{prop:mu-support} and lemma~\ref{lem:er}.
\end{proof}	

\subsection{Existence of invariant distribution for $X$}

In the next result, we identify the invariant distributions of the solution $X$ to \eqref{eq:EDSL} by proving that it is ergodic using the Harris recurrence property of the enclosed chain.

\begin{proof}[Proof of theorem~\ref{thm:existence-uniq-inv-c0}]
	From lemma~\ref{lemma:ET1}, $\mu^{inv}$ defined in \eqref{eq:muinv} is a probability measure. By  theorem~\ref{thm:existence-uniq-inv2}, we have existence of a unique invariant distribution for the Markov chain $(\bfw_n, S_n)_n$ which is given by $\mu^{\bfw,S}(dw,ds) = p(s,w)\mu^{\bfw}(dw)ds$. We then consider $h\in\mathrm L^{\infty}(\Ew\times\R_+;\R_+)$. We note that for $t>0$
	\[
	\frac{1}{T_{N_t+1}}\int_0^{T_{N_t}}h(X_s)ds\leq \frac1t\int_0^th(X_s)ds\leq \frac{1}{T_{N_t}}\int_0^{T_{N_t+1}}h(X_s)ds.
	\]
	where $N_t$ is the number of jumps before time $t$; it is finite by proposition~\ref{prop:regularity}.
	From theorem~\ref{thm:existence-uniq-inv2}, $(\bfw_n, S_n)_n$ is Harris recurrent and thus ergodic (\cite[theorem~4.3~p.140]{revuz_markov_1984}) on $\Ew\times\R_+$. It follows that $T_n\to\infty$ \textit{a.s.} Also: $T_n/T_{n+1}\to1 $ \textit{a.s.} Indeed, from the ergodicity of $(\bfw_n,S_n)_n$, we find  $T_n/n\to\E_{\muW}(T_1)<\infty$ \textit{a.s.} Therefore, to study the ergodicity of X, it is enough to focus on $\frac{1}{T_{N_t+1}}\int_0^{T_{N_t+1}}h(X_s)ds$.
	
	We have $H:(w,s)\to\int_{0}^{s}h(\Phi^{r}(w)) dr\in\mathrm L^{1}_+(\mu^{\bfw,S})$ because $\mu^{\bf w,S}(H)\leq \norm{h}_\infty \E_{\muW}(T_1) <\infty$.
	\[
	\int_0^{T_{N_t+1}}h(X_r)dr = \sum\limits_{k=0}^{N_t}\int_{T_k}^{T_{k+1}}h(X_r)dr = \sum\limits_{k=0}^{N_t} \int_{0}^{S_k}h(\Phi^{r}(X_{T_k^+})) dr = \sum\limits_{k=0}^{N_t} H(X_{T_k^+}, S_k).
	\]
	From proposition~\ref{prop:regularity}, $N_t\to\infty$ \textit{a.s.} Hence, from the ergodicity of $(\bfw_n,S_n)_n$, we find
	\[
	\frac{1}{N_t+1}\int_0^{T_{N_t+1}}h(X_r)dr  \to \int\mu^{\bfw,S}(dw,ds) \int_{0}^{s}h(\Phi^{r}(w)) dr\quad\ \textit{a.s.}
	\]
	It follows that
	\[
	\frac1t\int_0^th(X_s)ds \to	\mu^{inv}(h), \quad \textit{a.s.}
	\]
	which is then true for $h\in\mathrm L^{\infty}(\Ew\times\R_+;\R)$. Then, X is ergodic and has a unique invariant measure. 
	The fact that $\CP_1$ does not belong to the support of $\mu^{inv}$ is a consequence of \eqref{eq:muinv} and the fact that $\CP_1$ is not reachable from $\CR$ by the deterministic flow.
	
\end{proof}

\section{Nonlinear invariant distributions}\label{section:NLDI}
We study the nonlinear invariant distributions of the SDE \eqref{eq:MKV}. We thus fix all parameters and write for $\I$ constant $\muW_\I$ (resp.  $\muW_I$) the invariant distribution for the external current $I+\I$ (resp. $I$).
 
\begin{lemma}\label{lemma:kstat}
	Under the assumptions of theorem~\ref{thm:existence-uniq-inv-c0}, any current $\I$ solution of $\I=J\mu^{inv}_\I(\lambda )$ associated to a (nonlinear) stationary distribution is solution of
	\begin{equation}\label{eq:lndi}
	\I \cdot\E_{\muW_\I}(T_1) = J.
	\end{equation}
\end{lemma}
\begin{proof}
If the solution $X$ of the SDE \eqref{eq:EDSL} is stationary, we find that $\I =J\mu^{inv}_\I(\lambda )$ or $\I\cdot \E_{\muW_\I}(T_1) = J \E_{\muW_\I}\int_0^{T_1} \lambda(X_s)ds$. The lemma then follows from
\[
\E_{\muW_\I}\int_0^{T_1} \lambda\circ\Phi^s\,ds = \E_{\muW_\I} \int_0^\infty \mathbf 1(u\leq T_1)\ \lambda\circ\Phi^u du = \int \mu_\I(dw)\int_0^\infty \lambda(\Phi^t(\vr,w))e^{-\Lambda(0,(\vr,w),t)}dt = 1.
\]
\end{proof}
\noindent
We re-write equation \eqref{eq:lndi} as:
\[J=\I\int \muW_\I(dw) \int_0^{\tee^\I} e^{-\Lambda^\I(t,(\vr,w))} dt  \]
and define $T_\I$ on $\CR$ by
\[
T_\I(w) \defi  \int_0^{\tee^\I(w)}e^{-\Lambda^\I(t,(\vr,w))} dt=\E\left(T_1^{(\vr,w)}\right)<\infty.
\]
We note that $T_\I(w)<\tee^\I(w)$. To show existence of invariant distributions, we use the intermediate value theorem which relies on the continuity of the following mapping on $\R_+$
\begin{equation*}
\I\to\muW_\I(T_\I).
\end{equation*}
To prove continuity, as $T_\I$ is unbounded, it is not enough to show the continuity of $\I\to\muW_\I$ for the total variation norm since it is (only) the dual norm of $\mathrm L^\infty(\Ew )$. But since $T_\I$ is bounded by an exponential thanks to lemma~\ref{lemma:ET1}, this hints at using a weighted norm as we detail now.

We denote by $\mathcal M(\Ew )$ the Banach space of bounded Radon measures on $\Ew $ with the strong topology associated to the total variation norm:

\[
\|\mu\|\defi\sup \left\{|\mu(h)| \st h \in \mathrm L^\infty(\Ew),\|h\|_\infty \leq 1\right\}.
\]
We follow the notations of \cite{canizo_harris-type_2023} and write $\mathcal N$ (resp. $\mathcal P$) the subspace of $\mathcal M(\Ew )$ with zero mass (resp. of probability measures).
We denote by $V$ the Lyapunov function $V:\Ew  \to [1,\infty)$ defined by  $V(w) = e^{r(w-w^*)}$ and we introduce the weighted total variation norm on $\mathcal M(\Ew )$ defined by
\[
\norm{\mu}_V \defi \int V|\mu|.
\]
We note that $\norm{\mu}_V=\sup\limits_{\norm{h}_{\mathrm L^\infty}\leq 1}|\mu(h V)|$.
We denote by $\mathcal M_V$ the space $(\mathcal M(\Ew ), \norm{\cdot}_V)$.
We shall prove this continuity result in the quite restrictive case where $I\in \Iempty$ (see definition~\ref{def:I0}).

\begin{proposition}\label{prop:di-c0}
 	Grant the assumptions of proposition~\ref{prop-doebolin2}. The invariant distribution $\mu_I^\bfw$ is a continuous function of $I\in\Iempty$ in $\mathcal M_V$.
\end{proposition}
\begin{proof}
	We adapt the proof of Harris theorem from \cite{canizo_harris-type_2023} and  \cite{lacroix_chaines_2002}.
	We consider $ [\underline{I},\bar I]\subset\Iempty$ and a family of partitions $(\CP^I)_{I\in{[\underline{I},\bar I]}}$ from lemma~\ref{lemm_partition}.
	We introduce the mapping $F$ on $\mathcal M(\Ew )\times [\underline{I},\bar I]$ defined by \[F(\mu, I)\defi (\mU_I^*)^{2k_D}\mu\] where $\mU_I^*$ is the adjoint of $\mU_I$.
	
	Let $\vertiii{\mu}_V \defi \norm{\mu}+\beta\norm{\mu}_V$ be a norm for some $\beta>0$ to be adjusted later. Note that $\norm{\cdot}_V$ and $\vertiii{\cdot}_V$ are equivalent because $V\geq 1$.
	\begin{lemma}
	$F(\cdot, I)$ is contracting on $(\mathcal P, \vertiii{\cdot}_V)$ with Lipschitz constant $p\in(0,1)$ independent of $I$.		
	\end{lemma}
	\begin{proof}
	\textbf{Step 1.} $F(\cdot, I)$ leaves $\mathcal M_V$ invariant. From lemma~\ref{lem:er}, we have that $\forall I\in [0,\bar I]$, $\mU_I V\leq \gamma_L V+K_L$ where the Lyapunov constant $\gamma_L\in(0,1)$ and $K_L\geq0$ are independent of $I$. We then have
	\[\mU_I^{2k_D} V\leq \gamma_L^{2k_D} V+\sum_{n=0}^{2k_D-1}\gamma_L^nK_L :=\gamma_L^{2k_D} V+\tilde K_L.\]
	This implies that (using the Hahn-Jordan decomposition of $\mu$) \[\forall\mu\in\mathcal M_V,\quad \norm{(\mU_I^*)^{2k_D}\mu}_V\leq \gamma_L^{2k_D} \norm{\mu}_V+\tilde K_L\norm{\mu}\] which proves the point.

	\textbf{Step 2.} $F(\cdot, I)$ is contracting. Proposition~\ref{prop-doebolin2} and the proof of theorem 2.1 in \cite{canizo_harris-type_2023} imply that 
	\[
		\forall\nu\in\mathcal N,\ \norm{(\mU_I^*)^{2k_D}\nu}\leq \gamma_H\norm{\nu}
	\] for $\gamma_H\defi 1-\beta_D\in(0,1)$. This implies that 
	\[
		\forall\nu\in\mathcal N\cap\mathcal M_V,\quad\vertiii{(\mU_I^*)^{2k_D}\nu}_V\leq (\gamma_H+\beta \tilde K_L)\norm{\nu}+\beta\gamma_L^{2k_D}\norm{\nu}_V\leq \gamma \vertiii{\nu}_V 
	\] 
	where $\gamma=\max(\gamma_H+\beta\tilde K_L, \beta\gamma_L^{2k_D})$ can be made smaller than $1$ by decreasing $\beta$. 
	\end{proof}
	
	This reproves that there is a unique invariant distribution $\muW_I$ in $(\mathcal P, \vertiii{\cdot}_V)$ which can be expressed as
	\begin{lemma}
		\[
		\muW_I = \nu_D^0+\sum\limits_{n=0}^\infty (\mU_I^*)^{2k_Dn}\left((\mU_I^*)^{2k_D}\nu_D^0-\nu_D^0\right).
		\]	
	\end{lemma}
	\begin{proof}
	The invariant distribution is written $\muW_I = \nu_D^0+\rho^I$ with $\rho^I\in\mathcal N\cap\mathcal M_V$ as $\nu_D^0\in\mathcal M_V$. It follows that $(Id -(\mU_I^*)^{2k_D})\rho^I = (\mU_I^*)^{2k_D}\nu_D^0-\nu_D^0$ which implies the lemma by the contracting property.
	\end{proof}  
	We now look at the continuity of the mapping $I\to\muW_I$ in $(\mathcal P(\Ew), \vertiii{\cdot}_V)$. 
	Using the previous absolutely  convergent series, the continuity is a consequence of the continuity of $I\to \mU_I^*\nu_D^0$ (see lemma~\ref{lemma:acv}).
	We consider a sequence $(I_n)_n$ in $[\underline{I},\bar I]$ which converges to $I$. We note that
	\[
	\vertiii{\mU_{I_n}^*\nu_D^0 - \mU_I^*\nu_D^0}_V\leq C\norm{\mU_{I_n}^*\nu_D^0 - \mU_I^*\nu_D^0}_V
	\]
	by equivalence of the norms.
	We then consider $h\in\mathrm L^{\infty}(\mathbf E_w)$ and get from proposition~\ref{prop:K} (because $I\in\Ireg$):
	\begin{multline*}
			\lvert \nu_D^0(\mU_{I_n}Vh - \mU_{I}Vh)\rvert \leq \|h\|_\infty\cdot \int \nu_D^0(dw_0)\int_{\Ew }\lvert K_I(w_0,w)-K_{I_n}(w_0,w)\rvert V(w+\wb)dw 
			\\\defi  \|h\|_\infty\cdot\langle \nu_D^0, \lvert K_I - K_{I_n}\lvert\cdot V(\cdot+\wb)\rangle.
	\end{multline*}
	Hence, we have shown that
	\begin{equation}\label{eq:c0mu}
	\norm{\mU_{I_n}^*\nu_D^0 - \mU_I^*\nu_D^0}_V\leq  \langle \nu_D^0, \lvert K_I - K_{I_n}\lvert\cdot V(\cdot+\wb)\rangle
	\end{equation}
 and  $I\to\muW_I$ is continuous if we are able to show that the right hand side of \eqref{eq:c0mu} converges to zero as $n\to\infty$.
 	\begin{lemma}\label{lemma:KiC0}
		For almost all $w_0\in \Ew $, $\int_{\Ew }\lvert K_{I_n}(w_0,w)-K_{I}(w_0,w)\rvert V(w+\wb)dw$ converges to zero as $n\to\infty$.
	\end{lemma}
	\begin{proof}
		The lemma is a direct consequence of the Brezis-Lieb theorem~\ref{th-brezis-lieb}.
	
		\textbf{Step 1.}
		By proposition~\ref{prop:KC0}, the integrand $K_{I_n}(w_0,w)$ converges  to $K_{I}(w_0,w)$ \textit{a.e.} when $I\in\mathbf  I_\emptyset$.
	
		\textbf{Step 2.} We have $K_I\cdot V(\cdot+\wb)=\mU_{I}V<\infty$ by lemma~\ref{lem:er}.
		
		\textbf{Step 3.} We show that $\mU_{I_n}V$ converges to $\mU_{I}V$ \textit{a.e.}.
		Indeed:
		\[
			\mU_{I_n}V(w_0)=\int_0^\infty V\left(\wb+\hat w^{I_n}_{w_0}(\tau)\right)e^{-\tau}d\tau
		\]
		by lemma~\ref{lemma:Utau} where $\hat w$ is a time change of $w$. 
	 The integrand is continuous in $I_n$.
	 From corollary~\ref{coro:wbound}, we know that there is $C>0$ such that for all $n$ large enough and $\forall\tau\geq0$, $\hat w^{I_n}_{w_0}(\tau)\leq C+|w_0|$, this provides a domination of the integrand using the monotony of $V$. This proves that $\mU_{I_n}V$ converges to $\mU_{I}V$ \textit{a.e.} by Lebesgue's theorem and concludes the lemma.
	\end{proof}

	Let us define $g_n(w_0) \defi \int_{\Ew }\lvert K_I(w_0,w)-K_{I_n}(w_0,w)\rvert V(w+\wb)dw$. The previous lemma shows that $g_n$ converges to zero \textit{a.e.}. By lemma~\ref{lem:er}, for $n$ large enough, $g_n(w_0)\leq 2(\gamma_LV(w_0)+K_L)$. Hence, by Lebesgue's dominated convergence theorem, the right hand side of \eqref{eq:c0mu}, which equals $\int g_n d\nu_D^0$, converges to zero. We thus have proved that the mapping $I\to\muW_I$ is continuous.
\end{proof}

\begin{remark}
Let us comment on the restrictive assumption in proposition~\ref{prop:di-c0} that $I$ must belong to $\Iempty$.  
The restrictive assumption stems from our estimate of $\norm{\mU_{I_n}^*\nu_D^0 - \mU_I^*\nu_D^0}_V$.

This comes from the proof of proposition~\ref{prop:KC0} which relies on the fact that deterministic solutions are nearby for similar parameter values, a property called structural stability \cite{hale_ordinary_1980}. However, when there are stationary points, this is not necessarily the case. Indeed, one can prove that the one with largest membrane potential is a saddle point for which there are nearby solutions which explode in finite times and other which converge to periodic solutions.
\end{remark}

\begin{proof}[Proof of theorem~\ref{thm-nlDI}]
	We consider $I\in\Iempty$. For a constant $\I\geq 0$, $I+\I\in \Iempty$.
	We define $f:\R_+\ni \I\to \muW_\I(T_\I)$, which depends on $I$, and consider a non-negative sequence $\I_n\to \I_0$. 
	One then finds \[\lvert f(\I_n)-f(\I)\rvert\leq \lvert \muW_{\I_n}(T_{\I_n}) -\muW_{\I}(T_{\I_n})\rvert+\lvert\muW_\I(T_{\I_n}-T_\I)\rvert.\] As $T_{\I_n}=O( V+1)$ by \eqref{eq:T1x}, the first term converges to zero by  proposition~\ref{prop:di-c0}. 
	We then note that the mean jump time $T_{\I_n}$ converges to $T_\I$ a.e. by lemma~\ref{lem-TkC0}.
	Using the same domination \eqref{eq:T1x} and Lebesgue's dominated convergence theorem, one finds that the second term tends to zero as well. Hence, $f$ is continuous.

	We now study the equation $J = \I \cdot f(\I)$ for finding the invariant measures. By the intermediate value theorem, there is a solution $\I$ to this equation when $J$ is small enough, \textit{i.e.}, belongs to $[0, \sup\limits_{\I\geq0}\I f(\I)]$.
\end{proof}

\section{Applications}\label{section:applications}
We give an application of the previous results to the AdExp model \cite{brette_adaptive_2005} for which $F(v) = e^v-av$ with $a>0$. In order to have existence of a separatrix (see lemma~\ref{lemma:separatrix}), we need $a>3$: this allows to define partitions $\CP$. In order to have a unique solution to the mean field equations \eqref{eq:MKV} using theorem~\ref{th-existence-MKV-delay}, we need to choose the rate function $\lambda$. For example, the choice $\lambda(v) = e^{x+K}$ with $K\geq \log(2)$ satisfies assumptions~\ref{as:ratepos-increasing},~\ref{as:C2},~\ref{as:r-lambda-F}.
For the stationary distribution of the isolated neuron \eqref{eq:EDSL} using theorem~\ref{th-mai-results-muW}, we look at the nullcline $V_{null}$ which solves \[F(V_{null}(w))-w+I=0.\]
For $w$ large, the part of the $v$-nullcline $V_{null}^-$ which is on the left of the minimum of the nullcline (for $v=0$) has asymptotic behavior
\[
V_{null}^-(w)\underset{+\infty}{\sim} -\frac{w}{a}
\]
which shows that the assumptions of theorems~\ref{th-mai-results-muW},~\ref{thm:existence-uniq-inv-c0} are satisfied. For the existence of invariant distributions for \eqref{eq:MKV}, we look numerically at a situation where \eqref{eq:micro-unique} has two equilibria whence $I\notin\Iempty$ (we cannot apply theorem~\ref{thm-nlDI}). We solve numerically $\I =J\mu^{inv}_\I(\lambda )$ using the numerical scheme of \cite{aymard_mean-field_2019} and numerical continuation \cite{veltz_bifurcationkitjl_2020} based on the Newton algorithm. Figure~\ref{fig:nldi} suggests that theorem~\ref{thm-nlDI} should be valid for $I\notin\Iempty$. The numerical scheme does not allow to draw conclusions regarding the boundedness of $\I/\mu^{inv}_\I(\lambda )$ for $\I$ large.

\begin{figure}[ht!]
	\centering
	\includegraphics[width=0.9\textwidth]{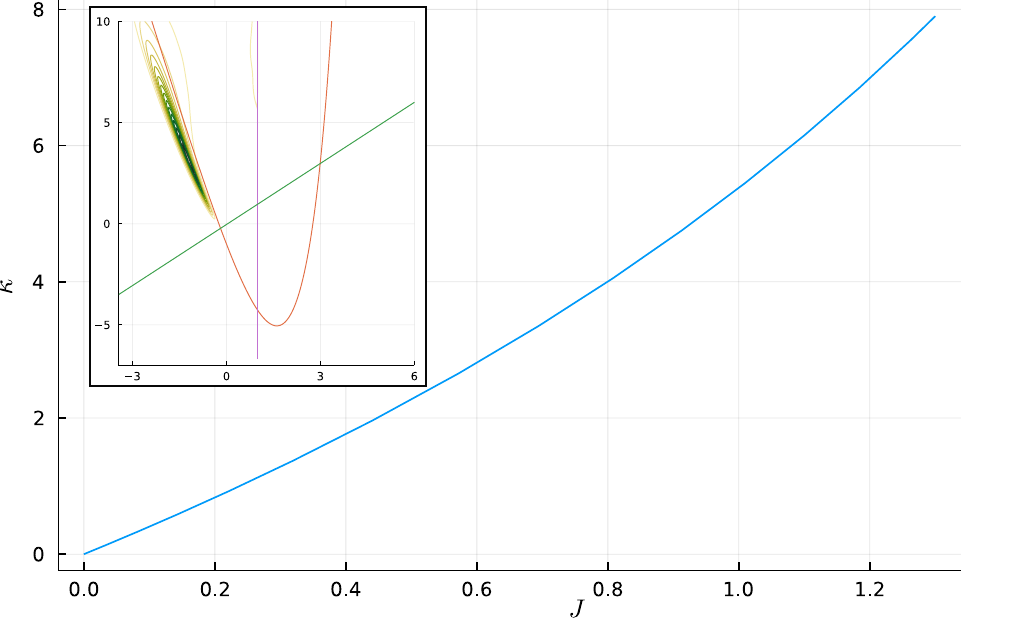}
	\caption{Plot of the solution $\I$ of $\I =J\mu^{inv}_\I(\lambda )$ as function of $J$. Inset: contour plot of the invariant distribution when $J=0$. Parameters: $F(v) = e^v-5v-2$, $\vr=1$, $\wb = 2.5$ and $\lambda(v) = e^{v+2}$. Numerical grid $[-7,8]\times[-10,30]$ with $1000\times 1000$ unknowns.}
	\label{fig:nldi}
\end{figure}

\section{Acknowledgments}
The author thanks Etienne Tanré and Quentin Cormier for fruitful discussions.
This work was supported by the grant ANR ChaMaNe, ANR-19-CE40-0024 and by the European Union’s Horizon 2020 Framework Programme for Research and Innovation under the Specific Grant Agreement No. 945539 (Human Brain Project SGA2, SGA3).

\appendix
\renewcommand{\theequation}{\thesection.\arabic{equation}}
\renewcommand{\thelem}{\thesection.\arabic{lem}}
\section{Functional analysis}
\begin{lemma}\label{lemma:acv}
	Let us consider a mapping $T_I\in\mathcal L(E)$ where $E$ is a Banach space and I belong to $\mathcal I$, an open ball of the reals. Assume that $\mathcal I\ni I\to T_I\cdot x_0$ is continuous for some $x_0\in E$. Further assume that there is $n\in\mathbb N^*$ such that for  for all $I\in \mathcal I$, $\norm{T_I^n}_{\mathcal L(E)}\leq\gamma<1$. Then $I\to (Id - T_I^n)^{-1}x_0$ is continuous.
\end{lemma}
\begin{proof}
	The inverse is well defined with expression
	\[
	y_I \defi (Id - T_I^n)^{-1}x_0 = \sum\limits_{p\geq 0} T_I^{p\cdot n}x_0.
	\]
	The series is absolutely convergent which implies the lemma.
\end{proof}
\begin{theorem}[Brezis-Lieb (\cite{brezis_relation_1983} remark (ii))]\label{th-brezis-lieb}
	Let $(E, \Sigma, \mu)$ be a measure space and consider a sequence of measurable mappings $f_n$.
	Assume that $f_n\to f$ \textit{a.e.}, that $\norm{f}_1<\infty$ and $\norm{f_n}_1\to\norm{f}_1$. Then $f_n\to f$ in $\mathrm L^1(E, \Sigma, \mu)$.
\end{theorem}

\section{ODE}
\begin{lemma}
	\label{lemma:separatrix}
	Grant assumption~\ref{hyp.F}-$(i)$. For $\I=0$, there are a constant $\alpha_-<0$ and a point $x_n\defi (v_-, v_-)\in\R^2$ on the $w$-nullcline such that:
	\[
	B_{\alpha_-,x_n}
	=
	\bigl\{(v,w)\in\R^2 \st  v\le v_-,\ w\le\alpha_-(v-v_-)+v_- \bigr\}
	\]
	is backward flow invariant, i.e. $\Phi^{-t}(B_{\alpha_-,x_n})\subset B_{\alpha_-,x_n}$ for all $t\geq 0$. 
\end{lemma}
\begin{proof}
	We adapt the proof of theorem 3.1 in \cite{touboul_spiking_2009}.
	To prove that $\Phi^{-t}$ leaves $B_{\alpha_-,x_n}$, $B$ in short, invariant, we determine under which conditions on $\alpha_-$ and $v_-$, the flow $\Phi^{-t}$ is inward for all $x\in\partial B$. The vector field associated with the backward flow is:
	$
	\Gb(x) \defi
	\bigl(-F(v)+w-I ,\, -v+w \bigr)
	$
	and the flow is inward at $x\in\partial B$ if and only if:
	\[
	\zeta \defi \vec n(x) \cdot \Gb(x) \geq 0
	\]
	where $\vec n(x)$ is an inward normal vector to the boundary at $x$.

	\smallskip

	The boundary $\partial B$ is composed of two segments: the first one is $\{v_-\}\times(-\infty,v_-]$, the second one is $\{(v,w)\,|\,  v\leq v_-,\ w=\alpha_-(v-v_-)+v_- \}$, see figure~\ref{fig:invariant.sets}.
	If $x$ belongs to the first segment for which $\vec n=(-1,0)$, we have:
	$\zeta
	=  F(v)-w +I$;
	if $x$ belongs to second segment for which $\vec n = (\alpha_-,-1)$, we have:
	$
	\zeta
	=
	\alpha_-\,(-F(v)+w-I) - (-v+w)$.
	So, we get the following two conditions for the  backward flow to be inward:
	\begin{align*}
		(i)  &\ \forall (v,w)\textrm{ such that } v=v_-,\,w\leq v_-:
		&& F(v)-w +I\geq 0,
		\\
		(ii) &\ \forall (v,w)\textrm{ such that } v\leq v_-,\,w = \alpha_-\,(v-v_-)+v_-:
		&& \alpha_-\,(-F(v)+w-I) - (-v+w)\geq 0.
	\end{align*}
	
	If $F(v_-)+I\geq v_-$ then condition $(i)$ is fulfilled: this amounts to choosing $(v_-,v_-)$ below the v-nullcline.  For the condition $(ii)$, plugging $w = \alpha_-\,(v-v_-)+v_-$ in the last inequality gives:
	\[
	(\alpha_-^2-\alpha_-+1)\,v-\alpha_-\,F(v)
	\geq (\alpha_-^2-\alpha_-)\,v_-+(1-\alpha_-)\,v_-+\alpha_-\,I.
	\]
	As $F$ is convex, one has $F(v)\geq F(v^*)+F'(v^*)(v-v^*)$ for arbitrary $v^*$. For $\alpha_-<0$, a sufficient condition is 
	\[
(\alpha_-^2-\alpha_-+1)\,v-\alpha_- (F(v^*)+F'(v^*)(v-v^*))
\geq (\alpha_-^2-\alpha_-) v_-+(1-\alpha_-) v_-+\alpha_- I.
\]
To have a decreasing slope, one needs $\alpha_-^2-\alpha_-+1-\alpha_- F'(v^*)<0$ which is possible if $F'(v^*)<-3$. Thanks to assumption~\ref{hyp.F}-$(i)$, one can choose $v^*$ accordingly. We then chose $v_-=v^*$ and check the condition for $v=v_-$, it gives
\[
-\alpha_-(F(v_-)+I-v_-)\geq 0
\]
which is satisfied by condition $(ii)$. The proof is complete.
\end{proof}
\begin{figure}[h!]
	\includegraphics[width=0.55\textwidth]{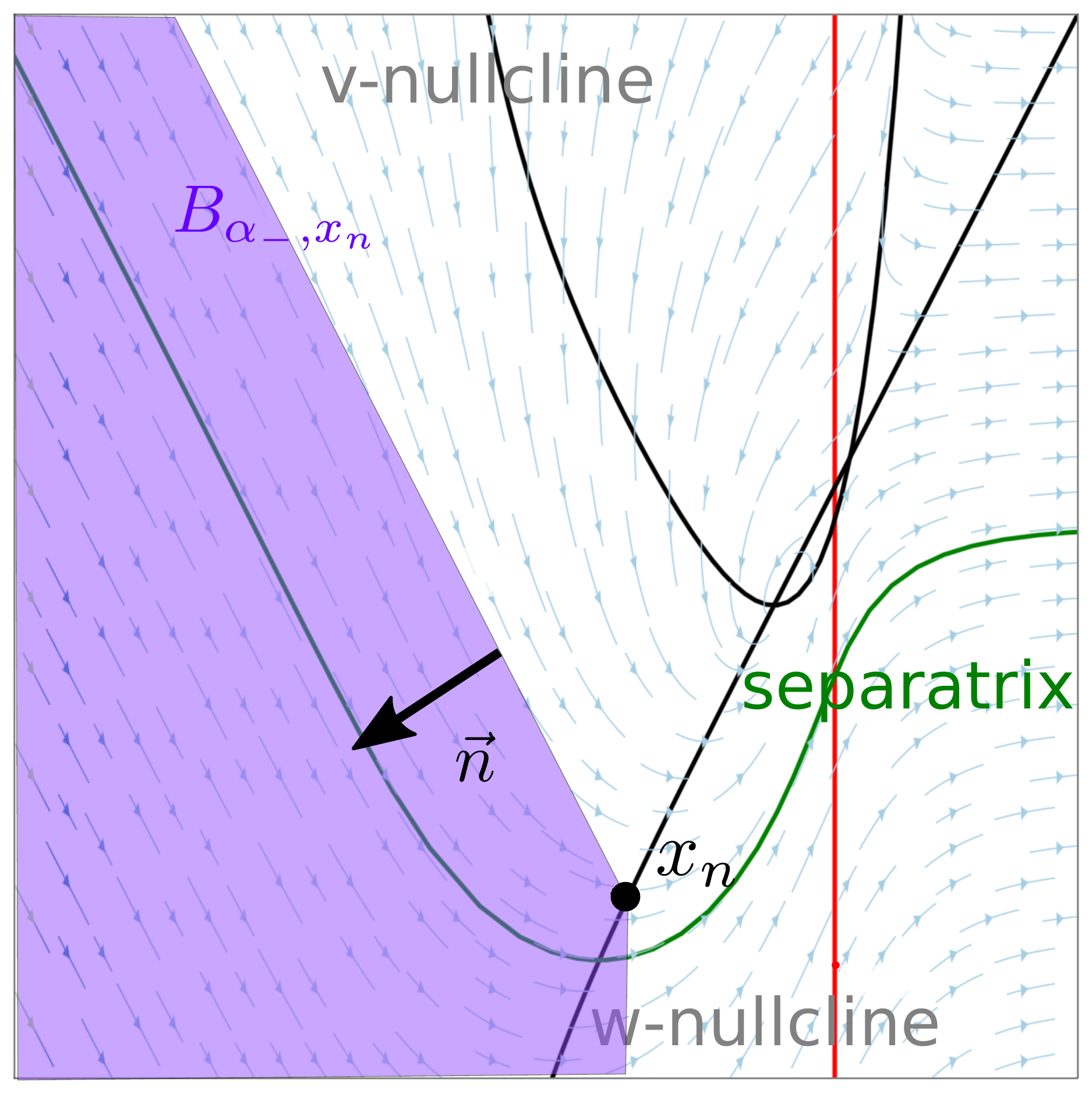}
	\caption{Drawing of the nullclines and different invariant sets involved in the proof of lemma~\ref{lemma:separatrix}.}
	\label{fig:invariant.sets}
\end{figure}

\begin{lemma}\label{lemma:whinf}
	Grant assumption~\ref{hyp.F} and consider a finite interval $[\underline I, \bar I]$ with a family of partitions $\CPf{\underline{I}}{\bar I}$ by lemma~\ref{lemm_partition} written $\mathcal P^I=(x_{sep}^I,w_{23})$. There is $w_{23}$ large enough such that the following holds for $\I=0$.
	\begin{enumerate}
		\item The solution from $(\vr,w_0)\in\CP_2^I$ hits the $v$-nullcline in finite time, we denote by $(v_n^I(w_0),w_n^I(w_0))$ the first intersection point.
		\item There is $\epsilon>0$ such that for all $w_0\geq w_{23},\ I\in [\underline I, \bar I]$: \[w_0-w_n^I(w_0)\geq \epsilon.\]
	\end{enumerate}
\end{lemma}
\begin{proof}
	\textbf{	Step 1.} The backward flow from $(v_{23}^I, w_{23})$ reaches the reset line in finite time $\tau_\CR^I$ and all points above on the left part of the nullcline do the same. Indeed, $t\to v_{(v_{23}^I, w_{23})}(-t)$ and $t\to w_{(v_{23}^I, w_{23})}(-t)$ are strictly increasing functions on $\R_+$. Hence, the hitting time $\tau^I_v(\vr, w_0)$ of the $v$-nullcline from $(\vr,w_0)$ is finite when $w_0\geq \tilde w_{23}$ with
	\[
		\tilde w_{23} = \max\limits_{I\in [\underline I, \bar I]} w_{(v_{23}^I, w_{23})}(-\tau_\CR^I)\geq w_{23}.
	\]
	We can take $\tilde w_{23}$ as new common boundary $\CP_2^I\cap\CP_3^I$. We denote it by $w_{23}$ here after.	
	
	\textbf{	Step 2.}
	We consider $\epsilon>0$ such that $\vr-\epsilon>\max\limits_{I}v_{23}^I$ (see figure~\ref{fig:setA}), By assumption~\ref{hyp.F}, the set $\mathbf L\defi \{(\vr-\epsilon,w_0),\ w_0\geq w_{23}\}$ is included in $\CP_{2}^I$. We denote by $(\vr-\epsilon,w^I_h(w_0))$ the first point on $\mathbf L$ that is hit by the trajectory from $(\vr,w_0)\in\CP_2^I$. We obviously have $w^I_n(w_0)\leq w^I_h(w_0)$ so to prove the lemma, it suffices to lower bound $w_0-w^I_h(w_0)$. For $\vr-\epsilon\leq v\leq \vr$, one finds:
\[
	\frac{d}{dv} W_{x}(v)=\frac{W_{x}(v)-v}{W_{x}(v)-F(v)-I}\geq \frac{w_h^I-\vr}{w_0-\underline F},
	\]
	where $x:=(\vr-\epsilon, w_h^I(w_0))$ and $\underline F \defi  \min\limits_{[\vr-\epsilon,\vr]}F+\underline  I$. Gronwall's lemma then gives 
\[
w_0 = W_x(\vr)\geq \tilde W_x(\vr) = w_h^I(w_0) + \epsilon\frac{w_h^I-\vr}{w_0-\underline F}
\]
and thus
\[
w_0-w_h^I(w_0)\geq  \epsilon\frac{w_h^I-\vr}{w_0-\underline F}
\]
or 
\[
\epsilon\frac{w_0-\vr}{w_0-\underline F} \leq \left(w_0-w_h^I(w_0)\right)\times\left[1+\epsilon\frac{1}{w_0-\underline F}\right] \Rightarrow C\epsilon\leq w_0-w_h^I(w_0)
\]
for some $C>0$ by optimizing the factors on $w_0$. This proves the lemma.
\end{proof}

\begin{proposition}\label{prop.tnbounded}
	Grant assumption~\ref{hyp.F} and let us consider a finite interval $[\underline I, \bar I]$ with a family of partitions $\CPf{\underline{I}}{\bar I}$ as in lemma~\ref{lemma:whinf}. Assume that $\I=0$.
	The hitting time $\tau^I_v(\vr, w_0)$ of the $v$-nullcline from $(\vr,w_0)\in\CP_2^I$ is bounded: there is $M>0$ such that \[\forall w_0\geq w_{23},\forall I\in[\underline{I}, \overline{I}]\ \tau^I_v(\vr, w_0)\leq M.\]
\end{proposition}
\begin{proof}
	We denote by $(v^I_n(w_0),w^I_n(w_0))$ the first point on the $v$-nullcline that is hit by the trajectory from $(\vr,w_0)\in\CP_2^I$. 

	We start by showing how  $w_0$ and $w^I_n$ control the hitting time $\tau^I_v(w_0)$. For $t\in[0,\tau^I_v(w_0)]$, we have  $\dot w_x(t)\leq \vr-w_x(t)$ with initial condition $x\in\CP_2^I\cap\CR$. Indeed, the $v$ component of the trajectory decreases above the $v$-nullcline. Using Gronwall's lemma, it implies that $\tau^I_v(\vr,w_0)\leq\log\left(\frac{w_0-\vr}{w^I_n-\vr}\right)\leq\frac{w_0-w^I_n}{w^I_n-\vr}$. Hence, to bound $\tau^I_v$, we shall control the ratio. 

	We now estimate $w^I_n$. Away from the $v$-nullcline, we can write $(v_{x}(t),w_{x}(t)) = (v_{x}(t),W_x(v_{x}(t))) $ where $W_x$ solves \eqref{eq:Wx}:
	\begin{equation}\label{eq:lemtau}
		\frac{dW_x}{dv}(v)=\frac{W_x(v)-v}{W_x(v)-F(v)-I},\quad W_x(x_1) = x_2. 
	\end{equation}
	$W_x$ is not well defined on the $v$-nullcline. We therefore consider a fixed $\epsilon>0$ small enough and define the first point \[{x_\epsilon^I}(w_0)\defi (v_n^{I,\epsilon}(w_0),w^I_n(w_0)+\epsilon)\] on the trajectory from $(\vr,w_0)$ which exists by lemma~\ref{lemma:whinf}.
	The time it takes for a trajectory from $x_\epsilon^I$ to reach the $v$-nullcline is bounded by $\frac{\epsilon}{w^I_n-\vr}$: this time is bounded as $w^I_n$ varies above $w_{23}$.

	We thus  aim at bounding the time for the backward flow to reach $\{v=\vr\}$. Note that $w_0 = W_{x_\epsilon^I}(\vr)$ and $W_{x_\epsilon^I}(v_n^{I,\epsilon}) = w^I_n+\epsilon$. To bound $\tau^I_v(\vr, w_0)$, we thus have to estimate $W_{x_\epsilon^I}(\vr)$. We use that
	\[
	\forall v\in[v_n^{I,\epsilon},\bar v],\quad 0\leq\frac{d}{dv} W_{x_\epsilon^I}\leq \frac{W_{x_\epsilon^I}-v_n^I}{W_{x_\epsilon^I}-w_n^I}.
	\]
	The bound on the numerator is straightforward. Let us focus on bounding the denominator. By construction of the partition, the $v$-nullcline is below $\{w=w_n^I\}$. Hence 
	\begin{multline*}
	\forall v\in[v_n^{I,\epsilon},\bar v],\quad W_{x_\epsilon^I}(v)-F(v)-I\geq W_{x_\epsilon^I}(v)-F(v_n^I)-I =W_{x_\epsilon^I}(v) - w_n^I\\(\geq w_n^{I,\epsilon}-w_n^I=\epsilon>0).
	\end{multline*}
	From the Gronwall's lemma, we obtain $\forall v\in[v_n^{I,\epsilon},\bar v],\quad W_{x_\epsilon^I}(v)\leq \tilde W(v)$ where $\tilde W$ is given by
	\begin{lemma} The solution of the ODE 
		\[
		\frac{d}{dv} \tilde W(v)= \frac{\tilde W(v)-v_n^I}{\tilde W(v)-w_n^I},\quad \tilde W(v_n^{I,\epsilon}) = w^I_n+\epsilon
		\]
		is
	\begin{equation}\label{eq:B29}
	\forall v\in[v_n^{I,\epsilon},\bar v],\quad W(v) \defi  {v_n^I} + \\(v_n^I-w_n^I)\cdot \mathcal W_{-1}\left[\frac{{\epsilon}-{v_n^I}+{w_n^I}}{{v_n^I}-{w_n^I}}\exp\left(\frac{{\epsilon}-v_n^{I, \epsilon}- {v_n^I}+{w_n^I}+v}{{v_n^I}-{w_n^I}}\right)\right].
	\end{equation}
	where $\mathcal W_{-1}$ is a branch of the Lambert function\footnote{It is such that $z=w e^w$ if and only if $w=\mathcal W(z)$.}. 
\end{lemma}
\begin{proof}
	Let us write the ODE as
	\[
v-v_n^{I,\epsilon} = 	\int_{w_n^I+\epsilon}^{\tilde W}1+\frac{v_n^I-w_n^I}{x-v_n^I}dx = \tilde W-w_n^I-\epsilon + (v_n^I-w_n^I)\log\left(\frac{\tilde W-v_n^I}{w_n^I+\epsilon-v_n^I}\right).
	\]
	If we define $\alpha\defi w_n^I-v_n^I>0$, $z\defi \tilde W-v_n^I\geq0$ and $\delta v\defi v-v_n^{I,\epsilon}\geq0$, we can re-write the equation as:
	\[
	\delta v+\alpha + \epsilon =z-\alpha \log\frac{z}{\alpha+\epsilon}
	\]
	which gives
	\[
		z
		= -\alpha
		\mathcal W_k\left(-\frac{\alpha+\epsilon}{\alpha}
		e^{-\frac{v+\alpha+\epsilon}{\alpha}}\right)
		,\quad k\in\{0,-1\}.
		\]
		
	\textbf{Step 1.}	
	For $z$ to be positive, the argument of the Lambert function must be in $[-e^{-1},0]$. Thus, defining $A = \frac{\alpha+\epsilon}{\alpha}
	e^{-\frac{v+\alpha+\epsilon}{\alpha}}$, we find $A\leq e^{-1}\,(1+\tfrac{\varepsilon}{\alpha})\,e^{-\frac{\varepsilon}{\alpha}}
	= e^{-1}\,f\!\left(\tfrac{\varepsilon}{\alpha}\right)$ where $f(x)\defi(1+x)e^{-x}\in[0,1]$ which proves the point.
	
   \textbf{Step 2.} There are two real valued branches of the Lambert function and we need to chose the correct one. $\mathcal W_0(u)\in(-1,0)$ whereas $\mathcal W_{-1}(u)\leq -1$ for $u\leq 0$. Let us show that $z\geq\alpha$ so that the $\mathcal W_{-1}$ yields the correct solution. This is so because $\tilde W-v_n^I\geq w_n^I-v_n^I=\alpha$.
\end{proof}
	
	We now show that $\frac{\tilde W(\vr)}{w^I_n} (\geq \frac{w_0}{w^I_n(w_0)})$ is bounded when $w_n^I\geq w_{23}$. 
	As $F$ is convex, we find that $\frac{|v^I_n|}{w^I_n} = \frac{|v^I_n|}{F(v^I_n)+I} $ is bounded as $w_n^I\to\infty$ and $I$ bounded. $|v_n^{I,\epsilon}|/w_n^I $ is also bounded because $|v_n^{I,\epsilon}(w_0)|\leq |v^I_n(w_0)|$ for $w_0$ large enough. We thus deduce from \eqref{eq:B29} that \[\forall w_0\geq w_{23},\ \forall I\in [\underline I, \bar I],\quad w^I_n(w_0)\leq A w_0\]
	for some constant $0<A<1$. It then implies that the time to reach $\{v=\vr\}$ from $x_\epsilon^I$ is bounded and so is $\tau^I_v$.
	
\end{proof}

\section{Properties of $p^\I$}
\noindent
Let us define
\begin{equation}\label{eq:C}
	C(x_1,\alpha, w^*) \defi \sup\limits_{v\geq x_1} C(x_1,v,\alpha, w^*)	, \quad \text{for }x_1\in\mathbb R
\end{equation}
where
\begin{equation*}
	\forall v\geq x_1,\alpha\geq 0,\ C(x_1,v,\alpha,w^*) \defi \lambda(v)\exp\left(-\int_{x_1}^{v}\frac{\lambda(u)}{F(u)-w^*+I+\alpha}du\right)
\end{equation*}
and $F(v)+I>w^*$ for $v\geq x_1$.

\begin{lemma}\label{lemma:p_on_vr}
Grant assumptions~\ref{hyp.F},\ref{as:ratepos-increasing},\ref{as:C} and assume that $\I\in C^0([S,T],\R_+)$ for $S<T<\infty$. Given a partition $\CP^I$, we have:
\[
	\forall S\leq s\leq t\leq T,\ \forall x\in\CR_{\CP^I},\quad p^\I(s,x,t)\leq \max(\lambda(v_{34}), C(v_{34},\norm{\I}_\infty,w^*(I)))
\]
where $C$ is defined in \eqref{eq:C} and $w^*(I)$ is the minimum of $\Ew^I$.
\end{lemma}
\begin{proof}
	Under assumption~\ref{as:ratepos-increasing}, $p^\I$ has the following expression
	\[
	p^\I(s,x,t)
	=
	\lambda(\Phi^{t}_s(x))\,e^{-\Lambda^\I(s,x,t)}\,\indic_{[s,\te{s,x,\I})}(t).
	\]
Given a partition $\CP^I=(x_{sep}^I,w_{23}^I)$, we chose $\tilde w_{23}$ valid for the current $I+\norm{\I}_\infty$. We then define $\tilde\CP^I\defi(x_{sep}^I,\tilde w_{23})$. Note that $\CR_{\CP^I} = \CR_{\tilde\CP^I}$. Hence, we shall use $\tilde\CP^I$.

We obtain the symbolic dynamics $\tilde\CP_2\cap\CR\to\tilde\CP_3\stackrel{possibly}{\to}\tilde\CP_4$ because $\tilde\CP_1$ cannot be reached from $x\in\tilde\CP_2\cap\CR$.
We write $\tau_i(s,x,\I)$ the hitting times of $\tilde\CP_i$ from $(s,x)$ for the deterministic flow.

\begin{itemize}
	\item We start with $x\in \tilde\CP_2\cap\CR$ and $s\leq t\leq \tau_3 $, one finds $p^\I(s,x,t)\leq \lambda(\vr)$ because $\lambda$ is non-decreasing.
	\item For $x\in \tilde\CP_3$ and $s\leq t\leq \tau_4$, one finds $p^\I(s,x,t)\leq \lambda(v_{34})$.
	\item For $x\in \tilde\CP_4$. Using the change of variable $v=v_x(s,t)$ for $s\leq t\leq \te{s,x,\I}\wedge T$, we find
	\[
	\int_s^t\lambda(v_x(s,u))du = \int_{x_1}^{v_x(s,t)}\frac{\lambda(v)}{F(v)-W_x(v)+I+\I(\tau_{s,x}(v)))}dv\]
	where $\tau_{s,x}(v)\defi\inf\{t\geq s,\ v_x(t,s)=v\}$.
	As $w^*\leq x_2\leq W_x(v)\ (\leq v)$, it gives
	\[
	p^\I(s,x,t)\leq \lambda(v_x(s,t))\exp\left(-\int_{x_1}^{v_x(s,t)}\frac{\lambda(v)}{F(v)-w^*+I+\norm{\I}_\infty}dv\right)
	\]
	which implies by assumption~\ref{as:C}: \[\forall s\leq t \in T\wedge\te{s,x,\I},\ p^\I(s,x,t)\leq C\left(x_1,\norm{\I}_\infty,w^*(I)\right).\]
	This bound is valid for $T\geq t\geq \te{s,x,\I}$ as $T_1<\te{s,x,\I}$ \textit{a.s.} under assumption~\ref{as:ratepos-increasing}.
\end{itemize}
Finally, for $x\in\tilde\CP\setminus\tilde\CP_i$ and $h\geq 0$: 
\begin{align*}
p^\I(s,x,\tau_i(s,x,\I)+h) &= p^\I(\tau_i(s,x,\I), \Phi_s^{\tau_i(s,x,\I)}(x),\tau_i(s,x,\I)+h)e^{-\Lambda^\I(s,x,\tau_i)}\\
&\leq p^\I(\tau_i(s,x,\I), \Phi_s^{\tau_i(s,x,\I)}(x),\tau_i(s,x,\I)+h)
\end{align*}
which allows to conclude the proof by noting that $\vr\leq v_{34}$ whence \[\max(\lambda(\vr),\lambda(v_{34}), C(v_{34},\norm{a}_\infty,w^*(I))) = \max(\lambda(v_{34}), C(v_{34},\norm{a}_\infty,w^*(I))).\]
\end{proof}

\begin{lemma}\label{lemma:p_C0}
	Grant assumptions~\ref{hyp.F},\ref{as:ratepos-increasing},\ref{as:C0} and assume that $\I\in C^0([S,T],\R_+)$ for some $S\leq T<\infty$. We also consider a partition $\CP$. Then for all $x\in\CP$, the mapping  $t\to p^\I(s,x,t)$ is continuous on  $S\leq s\leq t\leq T$.
\end{lemma}
\begin{proof}
	Under assumption~\ref{as:ratepos-increasing}, $p^\I$ has the following expression
	\[
	p^\I(s,x,t)
	=
	\lambda(\Phi^{t}_s(x))\,e^{-\Lambda^\I(s,x,t)}\,\indic_{[s,\te{s,x,\I})}(t).
	\]
	The only possible point of discontinuity  is $\te{s,x,\I}$ when it is finite. We thus assume that $S\leq \te{s,x,\I}\leq T <\infty$ and from theorem~\ref{th:jtb1}, we know that the trajectory from $x$ is unbounded and enter $\CP_4$. From the proof of lemma~\ref{lemma:p_on_vr}, we have for  $x\in\CP_4$:
	\[
	p^\I(s,x,t)\leq \lambda(v_x(s,t))\exp\left(-\int_{x_1}^{v_x(s,t)}\frac{\lambda(v)}{F(v)-w^*+I+\norm{\I}_\infty}dv\right)
	\]
	which  tends to zero as $t\to\te{s,x,\I}$ under assumption~\ref{as:C0} because $v_x(s,t)\to\infty$. The case $x\in\CP\setminus\CP_4$ is handled as in the proof of the previous lemma~\ref{lemma:p_on_vr}, this completes the proof.
\end{proof}
 
\section{Regularity of the transition kernel $K_I$}
\begin{proposition}\label{prop:KC0}
	Grant assumptions~\ref{hyp.F},\ref{as:ratepos-increasing}.
	Consider a finite interval $[\underline{I}, \bar I]\subset \Iempty$ (see definition~\ref{def:I0}) and a family of partitions $\CPf{\underline{I}}{\bar I}$ by lemma~\ref{lemm_partition}.
	Then, for almost all $w_0,w\in \Ew $, $I\to K_I(w_0,w)$ is continuous on $[\underline{I}, \bar I]$.
\end{proposition}
\begin{proof}
	We consider $I,I_0\in[\underline{I}, \bar I]$.
	Recall from the proof of proposition~\ref{prop:K} that $K_I = \sum\limits_{n=0}^{N_h(w_0,I)}K_{I,n}$ with
		\[
	 K_{I,n}(w_0,w)  \defi \frac{e^{-\Lambda(w_0,\tau_n)}	 }{|w-V^{(n)}(w,I)|}\lambda\left(V^{(n)}(w,I)\right)\exp\left(-\int_{w_n}^{w}\frac{\lambda(V^{(n)}(u,I))}{|u-V^{(n)}(u,I)|}du\right)\mathbf 1_{A_n}(w)
	\]
	where $w\neq w^{(n)}(w_0,I)$.
	When $I\in\Iempty$, we find $N_h(w_0,I)\leq 1$.
	Indeed, when there are no equilibria, all solutions explode in finite time. Hence, the solution from $(\vr,w_0)$ is unbounded and $N_h(w_0,I_0)\leq 1$. 
	$N_h(w_0,I_0)=0$ for all $w_0\leq \vr$ (below the $w$-nullcline	) and $N_h(w_0,I_0)=1$ otherwise.
	This shows that $N_h(w_0,I)$ is locally constant in $I$.
	Hence, for $I$ close enough to $I_0$, there are $N_h(w_0,I_0)$ terms in the sum defining $K_I(w_0,\cdot)$.
		
	We note that by the implicit functions theorem, $\tau_n, w^{(n)}, v^{(n/2)}$ are $C^1$ functions of $(w_0,I)$. Hence, for almost all $w,w_0\in\mathbf E_w$, for all $0\leq n\leq N_h(w_0,I_0)$, the function $I\to K_{I,n}(w,w_0)$ is continuous and the lemma is proved.
\end{proof}

\section{Alternative expression for $\mU_\I$}
We define a time change of \eqref{eq:micro-unique} when $\I$ is constant so that the unique solution is defined on $\mathbb{R}_+$ (i.e. no finite time explosion). Let us introduce 
\[
	\tau_{x}(t) \defi  \int_{0}^{t} \lambda\left(v_{x}(u)\right) d u
\]
which is well defined and invertible for $\lambda\in C^0(\R,\R_+^{*})$.
Then $\hat{\Phi}^{\tau}(x) \defi \Phi^{\tau_{x}^{-1}(\tau)}(x)$ solves

\begin{equation}\label{eq:timechange}
\left\{\begin{aligned}
	\frac{d}{d \tau} \hat{v}_x(\tau) &=\frac{F(\hat{v}_x(\tau))-\hat{w}_x(\tau)+I+\I}{\lambda(\hat{v}_x(\tau))},\\
	\frac{d}{d \tau} \hat{w}_x(\tau) &=\frac{\hat{v}_x(\tau)-\hat{w}_x(\tau)}{\lambda(\hat{v}_x(\tau))}, \\
	\hat{v}_x(0) &=x_{1}, \hat{w}_x(0)=x_{2}.
\end{aligned}\right.
\end{equation}

\begin{lemma}\label{lemma:Utau}
	Grant assumptions~\ref{hyp.F},\ref{as:ratepos-increasing}. Assume that $\lambda$ is continuous and $\I\in C^0(\R_+,\R_+)$. Then, the solution of the ODE \eqref{eq:timechange} is defined globally. For all $h\in\mathrm L^{\infty}(\mathbf E_w, \R_+)$, one finds
	\[
		\mU_\I h(w_0) = \int_0^\infty h(\wb + \hat{w}_{(\vr,w_0)}(\tau))e^{-\tau}d\tau.
	\]
\end{lemma}
\begin{proof}
	The solution is a time change of $(v_x(t), w_x(t))$. By contradiction, if a trajectory explodes in finite time, one has 
	\[
	\frac{d}{d \tau} \hat{v}_x(\tau) \leq\frac{F(\hat{v}_x(\tau))-w^*+I+\I}{\lambda(\hat{v}_x(\tau))}
	\]
	in $\CP_4$.
	By Gronwall's lemma, $\hat{v}_x(\tau)$ is bounded by the solution of 
	\[
	\frac{d}{d \tau} \hat{\mathbf v}(\tau)=\frac{F(\hat{\mathbf v}(\tau))-w^*+I+\I}{\lambda(\hat{\mathbf v}(\tau))}
	\]
	which is defined globally by assumption~\ref{as:ratepos-increasing} and this provides a contradiction.
	The second part follows from the change of variable $t=\tau_x^{-1}(\tau)$.
\end{proof}

\begin{lemma}\label{lem-TkC0}
	Grant assumptions~\ref{hyp.F},\ref{as:ratepos-increasing} and assume that $\lambda$ is continuous.
	Then, for all $x\in \CR_\CP$, $I\to \mathbb E_x(T_1^I)$ is continuous.
\end{lemma}
\begin{proof}
	Using lemma~\ref{lemma:Utau}, one obtains the following formula which removes the problematic $\tee$:
	\[
	T_I(x) = \int_0^\infty \frac{e^{-\tau}}{\lambda(\hat v_x^I(\tau))}d\tau.
	\]
	The integrand is continuous in $I$ for all $\tau.$
	We now consider an open bounded interval $B$ around $I_0$ and a family of partitions $(\CP^I)_{I\in B}$ from lemma~\ref{lemm_partition}. For all $x\in \CP_3^I$ and $\tau\geq 0$, one has $\lambda(\hat v_x^I(\tau))\geq \lambda(v_{23}^I)$. As the solutions are above the separatrix, one also finds that
	\[
		\lambda(\hat v_x^I(\tau))\geq \min(\lambda(V_{sep}(x_2,I)),\lambda(v^I_{23}))\stackrel{by\ continuity}{\geq}\alpha>0
	\]
	for some $\alpha$.
	The provides the domination for the use of Lebesgue's dominated convergence theorem which proves the lemma.
\end{proof}

\printbibliography

\end{document}